\documentclass[12pt]{article}

\usepackage{amsthm}
\usepackage{amsmath}
\usepackage{amssymb}
\input xypic
\usepackage[all,tips]{xy}
\usepackage{enumerate}
\usepackage[hidelinks]{hyperref}
\usepackage[margin=1in]{geometry}

\theoremstyle{plain}
\newtheorem{theorem}[subsubsection]{Theorem}
\newtheorem{lemma}[subsubsection]{Lemma}
\newtheorem{proposition}[subsubsection]{Proposition}
\newtheorem{corollary}[subsubsection]{Corollary}

\theoremstyle{definition}
\newtheorem{example}[subsubsection]{Example}
\newtheorem{remark}[subsubsection]{Remark}
\newtheorem{defin}[subsubsection]{Definition}

\newcommand{\bth}{\begin{thm}}
\renewcommand{\eth}{\end{thm}}
\newcommand{\bpr}{\begin{proposition}}
\newcommand{\epr}{\end{proposition}}
\newcommand{\ble}{\begin{lemma}}
\newcommand{\ele}{\end{lemma}}
\newcommand{\bco}{\begin{corollary}}
\newcommand{\eco}{\end{cororllary}}
\newcommand{\bdf}{\begin{defin}}
\newcommand{\edf}{\end{defin}}
\newcommand{\bex}{\begin{example}}
\newcommand{\eex}{\end{example}}
\newcommand{\bre}{\begin{remark}}
\newcommand{\ere}{\end{remark}}
\newcommand{\bcj}{\begin{conj}}
\newcommand{\ecj}{\end{conj}}

\newcommand{\beq}{\begin{equation}}
\newcommand{\eeq}{\end{equation}}

\newcommand{\ot}{{\otimes}}
\newcommand{\op}{{\oplus}}
\newcommand{\lb}{\label}

\newcommand{\bpf}{\begin{proof}}
\newcommand{\epf}{\end{proof}}

\newcommand{\bQ}{\mathbb{Q}}

\newcommand{\C}{\mathcal{C}}
\newcommand{\Z}{\mathcal{Z}}
\newcommand{\D}{\mathcal{D}}

\newcommand{\bZ}{{\mathbb Z}}

\newcommand{\G}{\mathcal{G}}

\newcommand{\Rep}{\mathcal{R}{\it ep}}

\newcommand{\Vect}{\mathcal{V}{\it ect}}
\newcommand{\Aut}{{Aut}}
\renewcommand{\d}{\partial}
\newcommand{\Hom}{{Hom}}
\newcommand{\V}{\mathcal{V}}
\newcommand{\Ex}{\mathcal{E}xt}
\newcommand{\Ext}{{Ext}}
\newcommand{\Map}{{Map}}
\newcommand{\Lex}{\mathcal{L}ex}
\newcommand{\Mex}{\mathcal{M}ex}
\renewcommand{\S}{\mathcal{S}}
\newcommand{\Id}{{Id}}

\newcommand{\Inv}{{Inv}}
\renewcommand{\dim}{{dim}}
\newcommand{\onehalf}{{\tiny \frac{1}{2}}}
\newcommand{\void}[1]{}

\everymath{\displaystyle}

\begin{document}

\title{Third Cohomology and Fusion Categories}
\author{Alexei Davydov and Darren A. Simmons}
\maketitle

\begin{abstract}
It was observed recently that for a fixed finite group $G$, the set of all Drinfeld centres of $G$ twisted by 3-cocycles form a group, the so-called group of {\em modular extensions} (of the representation category of $G$), which is isomorphic to the third cohomology group of $G$. We show that for an abelian $G$,  pointed twisted Drinfeld centres of $G$ form a subgroup of the group of modular extensions. We identify this subgroup with a group of quadratic extensions containing $G$ as a Lagrangian subgroup, the so-called group of {\em Lagrangian extensions} of $G$.  We compute the group of Lagrangian extensions, thereby providing an interpretation of the internal structure of the third cohomology group of an abelian $G$ in terms of fusion categories.  Our computations also allow us to describe associators of Lagrangian algebra in pointed braided fusion categories.
\end{abstract}

\tableofcontents

\section{Introduction}

In this paper, we are interested in a special class of monoidal categories---the class of fusion categories, i.e., $k$-linear, semi-simple, rigid monoidal categories with finitely many simple objects and finite-dimensional spaces of morphisms and such that the endomorphism algebra of the unit object coincides with the ground field $k$.  Due to the applications in representation theory, theoretical physics, and quantum computing the theory of fusion categories is experiencing a period of rapid development (see \cite{egno} and references therein).

Relations between low-dimensional group cohomology and monoidal categories in general (and fusion categories in particular) have a long history, going back to the beginning of at least the latter subject.  A model example is the fact that associativity constraints on the monoidal category of vector spaces graded by a group $G$ are nothing but 3-cocycles of $G$ with coefficients in the multiplicative group of the ground field. This observation goes back to S. Mac Lane, one of the founders of both subjects.

A more sophisticated example is the relation between constraints of a braided monoidal category on vector spaces graded by an abelian group $A$ (a so-called {\em braided pointed fusion category}) and the third abelian cohomology (see \cite{js}).  The structure of the braided monoidal category of $A$-graded vector spaces (the so-called braided pointed category) corresponds to a pair $(\alpha,c)$, where $\alpha$ is a 3-cocycle of $A$ (associativity constraint), and $c$ is a 2-cochain of $A$ (braiding). The hexagon axioms for the braiding are equivalent to certain equations on $(\alpha,c)$ known as the {\em abelian 3-cocycle condition}.  The Eilenberg--Mac Lane interpretation of the third abelian cohomology as the group of quadratic functions has a natural manifestation on the level of categories: equivalence classes of braided pointed fusion categories are labeled by pairs $(A,q)$, where $q:A\to k^*$ is the quadratic function computing the self-braiding $q(x) = c(x,x)$.

From the perspective of applications---especially in theoretical physics---the most interesting are braided fusion categories with a braiding as far from being symmetric as possible, i.e., the so-called {\em non-degenerate braided fusion categories} or {\em modular categories}.  For example, a braided pointed category is non-degenerate iff the kernel of the corresponding quadratic function is trivial. Another source of non-degenerate braided fusion categories is provided by the monoidal centre construction.  Applied to a fusion category $\S$ of a certain type, it produces a non-degenerate braided fusion category $\Z(\S)$.  The most studied examples of this type are {\em twisted Drinfeld centres}. The twisted Drinfeld centre $\Z(G,\alpha)$ of a finite group $G$ is the monoidal centre $\Z(\V(G,\alpha))$ of the category of $G$-graded vector spaces $\V(G,\alpha)$ with the associativity constraint defined by the 3-cocycle.  Recently, twisted Drinfeld centres of $G$ were interpreted as non-degenerate braided categories containing the representation category $\Rep(G)$ as the maximal symmetric subcategory. Moreover, the classes of such non-degenerate or modular extensions are shown to form a group (in \cite{lkw}), a fact which we use in this paper.

The main objects of study for us are Lagrangian algebras in non-degenerate braided categories.  The relative tensor product $\ot_R$ turns the category $\C_R$  of modules of a commutative algebra $R$ in a braided category $\C$ in to a tensor category.  The original braiding of $\C$ does not make the whole category $\C_R$ braided, it only works on a subcategory $\C_R^{loc}$ of so-called {\em local modules}.
A a commutative algebra $R$ is Lagrangian if the only local modules over $R$ are direct sums of $R$.  Lagrangian algebras play a special role in the theory of modular categories. Namely, a Lagrangian algebra $R\in\C$ allows us to identify $\C$ with the monoidal centre $\Z(\C_R)$ of the category $\C_R$ of $R$-modules in $\C$.  For example, when all simple modules over a Lagrangian algebra are invertible, the category $\C$ can be identified with a twisted Drinfeld centre 
\[
\C\ \simeq \Z(G,\alpha)\ ,
\]
where $G$ is the group of (isomorphism classes of) invertible $R$-modules and $\alpha$ is the associator of the pointed category $\C_R$.

In this paper, we examine the case of non-degenerate braided pointed categories.  In this case, all Lagrangian algebras correspond to Lagrangian subgroups of the grading group.  Moreover, all simple modules over Lagrangian algebras are invertible and the group of invertible modules is the quotient by the corresponding Lagrangian subgroup.  Thus a Lagrangian subgroup $L\subset A$ gives rise to a braided equivalence
\[
\C(A,q)\ \simeq\ \Z(A/L,\beta)\ .
\]
The primary purpose of the paper is to describe the associator $\beta\in H^3(A/L,k^*)$ as a function of $A, q$ and $L$.  We derive a formula for (a cocycle representing) it (see Appendix \ref{la}): 
\[
\beta(x,y,z) = 
\]
\[
\frac{\alpha{\left(s(x),s(y),s(z)\right)}\alpha{\left(\gamma(x,y),s(x+y),s(z)\right)}c{\left(s(x+y)+s(z),\gamma(x,y)\right)}\eta{\left(\gamma(x,y+z),\gamma(y,z)\right)}}{\alpha{\left(s(x),s(y+z),\gamma(y,z)\right)}c{\left(s(x+y),\gamma(x,y)\right)}\eta{\left(\gamma(x+y,z),\gamma(x,y)\right)}},
\]
where $s:A/L\to A$ is a section of the canonical projection, $\gamma$ is a $2$-cocycle representing the class of the abelian group extension $L\to A\to A/L$, $\eta\in C^2(L,k^*)$ is such that $\d\eta=\alpha|_L$, and $c\in C^2(A,k^*)$ is the coefficient for the braiding in the ambient category.

The above explicit expression is not very easy to use.  For example, the formula does not immediately tell when the cohomology class of the associator is nontrivial.  In order to find a more useful description of the associator, we look at the correspondence between triples $(A,q,L)$ (a non-degenerate quadratic group and a Lagrangian subgroup) and the cohomology classes in $H^3(A/L,k^*)$ from a different perspective. By fixing $L$ and varying $(A,q)$ we turn this correspondence into a homomorphism of groups. Namely we define a natural group structure on the set of isomorphism classes of Lagrangian extensions of $L$ (non-degenerate quadratic group $(A,q)$ containing $L$ as a Lagrangian subgroup). Assigning the associator to the triple $(A,q,L)$ now becomes a homomorphism from the group $Lex(L)$ of Lagrangian extensions to the third cohomology
\begin{equation}\lb{main}\tag{$*$}Lex(L)\ \to \ H^3(\widehat{L},k^*)\ .\end{equation}
Here, we use the natural identification of $A/L$ with the character group $\widehat{L}=Hom(L,k^*)$ induced by the quadratic function on $A$.  The homomorphism \eqref{main} is a convenient way of expressing the associator $\beta\in H^3(A/L,k^*)$ as a function of $A, q$ and $L$.

More precisely, according to \cite{lkw}, we can identify $H^3(\widehat{L},k^*)$ with the group of modular extensions of the representation category $\Rep(\widehat{L}) = \V(L)$. Under this identification the subgroup of pointed modular extensions corresponds to $Lex(L)$.
The homomorphism \eqref{main} fits into a short exact sequence
\[
\xymatrix{1\ar[r] & Lex(L) \ar[r] & H^3(\widehat{L},k^*)\ar[r]^(0.425){{Alt}_3}&\Hom(\Lambda^3\widehat{L},k^*)\ar[r]&1}\ ,
\]
i.e., $Lex(L)$ is the kernel of the alternation map ${Alt}_3: H^3(\widehat{L},k^*)\to\Hom(\Lambda^3\widehat{L},k^*)$. We denote by $\Hom(\Lambda^nB,M)$ the group of alternating $n$-linear forms on $B$ with values in an abelian group $M$.
Assigning to a Lagrangian extension $L\subset A$ the class of a short exact sequence $0\to L\to A\to \widehat{L}\to 1$ gives a homomorphism of groups 
\[
\phi:Lex(L)\ \to\ \Ext(\widehat{L},L)^\tau\ .
\]
Here, $\Ext(\widehat{L},L)^\tau$ is the subgroup of invariants of the involution $\tau:\Ext(\widehat{L},L)\to \Ext(\widehat{L},L)$, given by taking the dual.

%When $L$ is 2-torsion-free, the map $\phi$ is an isomorphism.  
For a general finite abelian $L$, the situation is a bit more complicated. We have an exact sequence:
$$\xymatrix{0 \ar[r] & J(L) \ar[r] & Lex(L)\ar[r]^(.45)\phi & \Ext(\widehat{L},L)^\tau\ar[r] & L_2 \ar[r] & 0}$$
The functor $J$ is an additive endofunctor on the category of finite abelian groups such that for a prime $p$
$$J(\bZ/p^l\bZ) = \left\{\begin{array}{cc}\bZ/2\bZ ,& p=2\\ 0 ,& p\ \text{odd}\end{array}\right. $$
The functor $J$ sends isomorphisms between cyclic 2-groups to the identity and non-isomorphisms to zero. \\

Altogether these give a decomposition of degree-three cohomology into natural polynomial functors of degrees three, two, and one: 
$$
\xymatrix
{
&0\ar[d]&&&\\
&J(\widehat{B})\ar[d] &&&\\
0\ar[r]&Lex(\widehat{B})\ar[d]^{\phi}\ar[r]&H^3(B,\bQ/\bZ)\ar[rr]^(0.425){{Alt}_3}&&\Hom(\Lambda^3B,\bQ/\bZ)\ar[r]&0\\
&\Ext(B,\widehat{B})^{\tau} \ar[d]\\
& (\widehat{B})_2 \ar[d]&&&\\
& 0
}
$$
%Here, we write $\widehat{L}=B$.  
Since our ground field $k$ is algebraically closed of characteristic zero, we can trade $k^*$ for the universal torsion group $\bQ/\bZ$ as the coefficient group for cohomology of a finite $B$.
Note that the answers for cohomology in degree one and two are much easier:
\[
H^1(B,\bQ/\bZ) = \widehat B,\qquad H^2(B,\bQ/\bZ) = \Hom(\Lambda^2B,\bQ/\bZ)\ .
\]
The rise in complexity in degree three reflects the more involved decomposition of the homology $H_3(B)$ in terms of natural polynomial functors (see \cite{bre}).
 
We will deal with fusion categories.  We will use the language of categorical groups  (see, e.g. \cite[Section 3.1]{js}). This language will be very helpful in comparing the groups $Lex(L)$ and $Ext(\widehat{L},L)$.  A {\em categorical group} is a monoidal category $\G$ in which all morphisms and all objects are invertible. The {\em standard invariants} of $\G$ are the zeroth and first homotopy groups $\pi_0(\G)$ and $\pi_1(\G)$;  $\pi_0(\G)$ is the group of isomorphism classes of objects of $\G$, and $\pi_1(\G)$ is the group of automorphisms of the unit object of $\G$.

\begin{remark}\lb{acg}
For any $X\in\G$, there is an isomorphism $\Aut_\G(I)\to\Aut_\G(X)$ defined as follows.  For $a\in\Aut_\G(I)$, define $\alpha\in\Aut_\G(X)$ by the diagram
\[
\xymatrix
{
I\ot X\ar[r]^{\lambda_X}\ar[d]_{a\ot 1}&X\ar[d]^{\alpha}\\
I\ot X\ar[r]_{\lambda_X}&X\\
}
\]
Here, $\lambda$ denotes the left unit isomorphism in $\G$. Let $g_X:\Aut_\G(X)\to\Aut_\G(I)$ denote the inverse of this map. Note that, for any $Y\in \G$ and any morphism $x:X\to Y$, we have $g_X(a)=g_{Y}(xax^{-1})$.
\end{remark}

Throughout this paper, we will adhere to the following notation conventions:  lower-case Latin and Greek characters are used for elements and morphisms, upper-case Latin and Greek characters for objects and functors, calligraphic Latin characters for categories and 2-functors, boldface Latin characters for 2-categories, and boldface Fraktur characters for 3-categories.

\section*{Acknowledgment}

The authors would like to thank Larry Breen, Ronnie Brown, and the anonymous referee for valuable comments. 

%%%%%%%%%%%%%%%%%%%%%

\section{Drinfeld centres of finite groups}

In this section, we follow \cite{lkw} to interpret $H^3(G,k^*)$ as the group of modular extensions of the representation category $\Rep(G)$, and then we consider pointed modular extensions of $\Rep(G)$.

\subsection{The categorical group $\Mex(\mathcal{S})$}

Recall (e.g. from \cite{dmno}) that a braided fusion category $\mathcal{D}$ is {\em non-degenerate} if its {\em symmetric centre}
\[
\Z_{sym}(\D)=\{X\in\mathcal{D}\ |\ c_{X,Y}\circ c_{Y,X}=1_{Y\ot X}\ \forall Y\in\D\}
\]
coincides with $\Vect$.  Let $\mathcal{S}$ be a symmetric fusion category.  Following \cite{lkw}, we say that a non-degenerate braided fusion category $\mathcal{D}$ containing $\S$ as a full subcategory is a {\em modular extension}\footnote{We are abusing language here; what we call a ``modular extension" would more properly be termed a ``non-degenerate braided extension".} of $\S$ if the {\em symmetric centraliser}
\begin{equation}\lb{syce}\C_\mathcal{D}(\S)=\{X\in\mathcal{D}\ |\ c_{X,Y}\circ c_{Y,X}=1_{Y\ot X}\ \forall Y\in\S\}
\end{equation} 
coincides with $\S$.

\begin{remark}\lb{cme}
A non-degenerate braided category $\mathcal{D}$ containing $\S$ as a full subcategory is a modular extension if and only if $\dim(\D) = (\dim(\S))^2$.
Indeed, $\S$ is always a subcategory of its symmetric centraliser $\C_\mathcal{D}(\S)$, and the relation between dimensions $\dim(\mathcal{D})=\dim(\S)\cdot\dim(\C_{\mathcal{D}}(\S))$ makes $\dim(\C_{\mathcal{D}}(\S))=\dim(\S)$ equivalent to the condition $\S = \C_\mathcal{D}(\S)$.
\end{remark}

Define the category $\Mex(\S)$ as follows.  Objects of $\Mex(\S)$ are modular extensions of $\S$.  Morphisms are isomorphism classes of braided equivalences preserving the subcategory $\S$, i.e., making the diagram 
\[
\xymatrix{\C \ar[rr] && \D\\ & \S \ar[ul] \ar[ur] }
\]
commute on the nose.

\begin{remark}\lb{2cg}
If we consider isomorphisms between braided equivalences of modular extensions as 2-cells, we naturally end up with a 2-category ${\mathbb M}{\bf ex}(\S)$.
\end{remark}

To describe the monoidal structure on $\Mex(\S)$, we recall some basic facts about commutative algebras in braided fusion categories.

Following \cite{dmno}, we call a commutative algebra $R$ in a braided fusion category $\D$ {\em etale} if it is indecomposable and separable.  The category $\D_R$ of (right) $R$-modules over an etale algebra $R$ is fusion with respect to the relative tensor product $\ot_R$. The categorical dimension of $\D_R$ satisfies $dim(\D) = dim(\D_r)dim(R)$.
An $R$-module $M$ is said to be {\em local} if the following diagram commutes:
\[
\xymatrix{M\otimes R  \ar[r]^\nu \ar[d]_{c_{M,R}} & M\\ R\otimes M \ar[r]_{c_{R,M}} & M\otimes R \ar[u]_\nu}
\]
The full subcategory $\D^{loc}_R$ of local right $R$-modules is braided \cite{pa}. In terms of categorical dimension $dim(\D) = dim(\D^{loc}_R)dim(R)^2$.
An algebra $R$ in a braided fusion category $\D$ is {\em Lagrangian} if any local $R$-module is a direct sum of copies of the regular module $R$, i.e. $\C^{loc}_R$ is equivalent to the category $\Vect$ of vector spaces.  For a Lagrangian $R$, the natural braided tensor functor into the monoidal centre $\D\to \Z(\D_R)$ is an equivalence. Dimension-wise, we have $dim(\D) = dim(R)^2$.

For $\mathcal{D},\mathcal{D}'\in\Mex(\S)$, the Deligne product $\S\boxtimes\S$ is a full subcategory of $\mathcal{D}\boxtimes\mathcal{D}'$.  The tensor product functor $\ot:\S\boxtimes\S\to\S$ has a two-sided adjoint $F:\S\to\S\boxtimes\S$.  This gives rise to an etale algebra $R=F(I)\in\S\boxtimes\S$, where $I\in\S$ is the monoidal unit object.  The category of local modules $(\mathcal{D}\boxtimes\mathcal{D}')_{R}^{loc}$ is a modular extension of $\S$.  Indeed, $(\S\boxtimes\S)_{R}^{loc} = (\S\boxtimes\S)_{R} = \S$, so that $(\mathcal{D}\boxtimes\mathcal{D}')_{R}^{loc}$ is an extension of $\S$. Since 
$\dim(R)=\dim(\S)$, we have $\ \dim((\mathcal{D}\boxtimes\mathcal{D}')_{R}^{loc})=\dim(\mathcal{D}\boxtimes\mathcal{D}')\dim(R)^{-2}=\dim(\S)^2$.
Define an operation 
\[
\odot_\S:\Mex(\S)\times\Mex(\S)\to\Mex(\S),\qquad \mathcal{D}\odot_{\S}\mathcal{D}'=(\mathcal{D}\boxtimes\mathcal{D}')_{R}^{loc}.
\]
The groupoid $\Mex(\S)$ is monoidal with respect to $\odot_{\S}$.  The monoidal unit is the monoidal (or Drinfeld) centre $\Z(\S)$ of $\S$, considered as a modular extension of $\S$ with respect to the natural embedding $\S\to\Z(\S)$.
Moreover, $\Mex(\S)$ is a categorical group.  Indeed, the quasi-inverse of a modular extension $\C$ with respect to $\odot_{\S}$ is $\overline\C$, i.e., the category $\C$ with the inverse braiding.

The zeroth homotopy group $\pi_0(\Mex(\S))=\text{Mex}(\S)$ was called in \cite{lkw} the group of {\em modular extensions} of $\S$.
By the definition, the first homotopy group $\pi_1(\Mex(\S))$ coincides with the group $\Aut_{br}(\Z(\S)/\S)$ of isomorphism classes of braided tensor autoequivalences of $\Z(\S)$ fixing objects of $\S$ on the nose.

\begin{remark}\lb{2hg}
The 2-category ${\mathbb M}{\bf ex}(\S)$ from remark \ref{2cg} is clearly a 2-categorical group.  The standard invariants $\pi_0$ and $\pi_1$ of ${\mathbb M}{\bf ex}(\S)$ are the same as for $\Mex(\S)$.  It follows from the definition that  $\pi_2(\mathbb{M}{\bf ex}(\S)) ={\ker}{\left(\Aut_\ot(Id_{\Z(\S)}) \to \Aut_\ot(Id_\S)\right)}$.
\end{remark}

We will be interested in the case when the base symmetric category $\S$ is the category $\Rep(G)$ of finite-dimensional representations of a finite group $G$.  Modular extensions of $\Rep(G)$ turn out to be twisted Drinfeld centres of $G$, which we are going to describe now.

Denote by $\V(G,\alpha)$ the category of $G$-graded vector spaces with the usual graded tensor product and associativity twisted by the 3-cocycle $\alpha\in Z^3(G,k^*)$: 
\begin{equation}\lb{ass}
\alpha_{V,U,W}(v\ot(u\ot w)) = \alpha(x,y,z)(v\ot u)\ot w,\qquad v\in V_x,\ u\in U_y,\ w\in W_w\ .
\end{equation}

An {\em $\alpha$-projective $G$-action} ({\em cf.} \cite[\textsection\ 2.1]{das1}) on a $G$-graded vector space $V$ is a collection of automorphisms 
$x:V\to V,\ v\mapsto x.v$ such that $x(V_y) = V_{xyx^{-1}}$, and
\begin{equation}\label{pa}
(xy).v = \alpha(x,y|z)x.(y.v),\qquad  v\in V_z\ .
\end{equation}
Here,
\[
\alpha(x,y|z) = \frac{\alpha(x,yzy^{-1},y)}{\alpha(x,y,z)\alpha(xyz(xy)^{-1},x,y)}
\]
Similarly, define
\[
\alpha(x|y,z) = \frac{\alpha(x,y,z)\alpha(xyx^{-1},xzx^{-1},x)}{\alpha(xyx^{-1},x,z)}\ .
\]
The following identities follow directly from the 3-cocycle condition for normalized $\alpha$:
\begin{align}\lb{pca}\begin{split}
\alpha(x,yz|w)\alpha(y,z|w) & =  \alpha(xy,z|w)\alpha(x,y|zwz^{-1})\ , \\
\alpha(xy|z,w)\alpha(x,y|z)\alpha(x,y|w) & =  \alpha(x,y|zw)\alpha(y|z,w)\alpha(x|yzy^{-1},ywy^{-1})\ ,\\
\alpha(y,z,w)\alpha(x|yz,w)\alpha(x|y,z) & =  \alpha(x|y,zw)\alpha(x|z,w)\alpha(xyx^{-1},xzx^{-1},xwx^{-1})\ .
\end{split}
\end{align}
Define the twisted Drinfeld centre $\Z(G,\alpha)$ as follows. Objects of $\Z(G,\alpha)$ are $G$-graded vector spaces  together with $\alpha$-projective $G$-action. Morphisms are grading and action-preserving homomorphisms of vector spaces.
The tensor product in $\Z(G,\alpha)$ is the tensor product of $G$-graded vector spaces, with $\alpha$-projective $G$-action defined by
\begin{equation}\label{tp}
x.(u\otimes v) = \alpha(x|y,z)(x.u\otimes x.v),\quad u\in U_x,\ v\in V_y\ .
\end{equation}
The associativity is given by \eqref{ass}.  The monoidal unit is $I=I_e=k$ with trivial $G$-action.  The braiding is given by
\begin{equation}\label{br}
c_{U,V}(u\otimes v) = x.v\otimes u,\quad u\in U_x, v\in V\ .
\end{equation}

Note that $\Z(G,\alpha)$ is equivalent to the monoidal centre $\Z(\V(G,\alpha))$ (see \cite{das1} for details).  Note also that $\Rep(G)$ is a full symmetric subcategory of $\Z(G,\alpha)$ (consisting of trivially graded objects).  Moreover, $\dim(\Z(G,\alpha)) = |G|^2 = (\dim(\Rep(G))^2$, so $\Z(G,\alpha)$ is a modular extension of $\Rep(G)$ by remark \ref{cme}.

The following was proved in \cite[Theorem 4.22]{lkw}. We add (a sketch of) a proof here.  
\begin{proposition}\lb{p0m}
The assignment $\alpha \mapsto \Z(G,\alpha)$ gives an isomorphism 
\[
H^3(G,k^*)\ \to\ {Mex}(\Rep(G))\ .
\] 
\end{proposition}
\begin{proof}
The homomorphism property of the assignment is established by a braided equivalence
\[
\Z(G,\alpha)\odot_{\Rep(G)}\Z(G,\beta)\ \to\ \Z(G,\alpha\beta)\ .
\]
Note that $\Z(G,\alpha)\boxtimes\Z(G,\beta)\simeq\Z(G\times G,\alpha\times\beta)$ as braided fusion categories.
Let $\delta:G\to G\times G$ be the diagonal embedding.
The tensor product functor $\ot:\Rep(G)\boxtimes\Rep(G)\to \Rep(G)$ corresponds to the inverse image functor $\delta^*:\Rep(G\times G)\to\Rep(G)$ 
upon the identification $Rep(G)\boxtimes\Rep(G) = \Rep(G\times G)$. By Frobenius reciprocity its adjoint $\delta_*$ is the induction with respect to $\delta$. In particular, $\delta_*(k)$ is the function algebra $k{\left((G\times G)/\delta(G)\right)}\in\Z(G\times G,\alpha\times\beta)$.  By \cite[Theorem 3.7]{das1}, its category of local modules is equivalent to $\Z(G,\alpha\beta)$.
Let now $\D$ be a modular extension of $\Rep(G)$.  The function algebra $A=k(G)$ with the regular $G$-action is etale in $\Rep(G)$ and in $\D$. Since $\dim(A)=\dim(\Rep(G))$, the algebra $A$ is Lagrangian.  Thus $\D\simeq \Z(\D_A)$. Moreover, the group of algebra automorphisms is $\Aut_{alg}(A)=G$. Hence the category of modules $\D_A$ is $G$-graded $\op_{g\in G}\D_A^{g\text{-}loc}$ into $g$-local modules and as the result is tensor equivalent to $\V(G,\alpha)$ for some $\alpha\in Z^3(G,k^*)$ (see \cite{ki} for details). Thus we have $\D=\Z(\D_A)\simeq\Z(\V(G,\alpha))\simeq\Z(G,\alpha)$ which shows that the assignment of the proposition is bijective.
\end{proof}

The first homotopy group $\pi_1(\Mex(\Rep(G)))$ also has a cohomological description.
\begin{lemma}\lb{p1m}
$\pi_1(\Mex(\Rep(G))) \simeq H^2(G,k^*)$.
\end{lemma}
\begin{proof}
$\pi_1(\Mex(\Rep(G)))$ is the group  $\Aut_{br}(\Z(\Rep(G))/\Rep(G))$ of braided autoequivalences of the monoidal centre $\Z(\Rep(G))$ that fix the full subcategory $\Rep(G)\subseteq\Z(\Rep(G))$ on the nose.  The monoidal centre $\Z(\Rep(G))$ is equivalent, as a braided monoidal category, to the untwisted Drinfeld centre $\Z(G)$.  The algebra of functions $k(G)\in \Rep(G)$ is an etale algebra in $\Z(G)$; its category of right modules $\Z(G)_{k(G)}$ is equivalent to the category $\V(G)$ of $G$-graded vector spaces.  This equivalence gives rise to a homomorphism  $\Aut_{br}(\Z(\Rep(G))/\Rep(G))\to \Aut_{\ot}(\V(G))$.  It follows from \cite[Corollary 6.9]{nr} that it is an isomorphism with the group $\Aut_{\ot}^{1}(\V(G))$ of isomorphism classes of tensor structures on the identity functor (the so-called {\em soft} autoequivalences). The group $\Aut_{\ot}^1(\V(G))$ is, in turn, isomorphic to $H^2(G,k^*)$ by \cite[Proposition 2.5]{da}.
\end{proof}

\begin{remark}
The group $\pi_2({\mathbb M}{\bf ex}(\Rep(G)))$ is isomorphic to the first cohomology group $H^1(G,k^*)$.  Indeed, according to remark \ref{2hg}, the group $\pi_2({\mathbb M}{\bf ex}(\Rep(G)))$ coincides with $\ker(\Aut_\ot(Id_{\Z(\Rep(G))}) \to \Aut_\ot(Id_{\Rep(G)}))$.\\
As was pointed out to us by L. Breen, the 2-categorical group ${\mathbb M}{\bf ex}(\Rep(G))$ is the truncation of a 3-categorical group ${\bf\mathfrak{Mex}}(\Rep(G))$, with the third homotopy group $\pi_3({\bf\mathfrak{Mex}}(\Rep(G)))$ coinciding with the zeroth cohomology group $H^0(G,k^*)$.
\end{remark}

We end this section with a crossed-module representing the categorical group $\Mex(\Rep(G))$. Denote by $(C^*(G,k^*),\partial)$ the standard complex computing the cohomology of $G$ with coefficients in $k^*$.
\begin{lemma}
The categorical group $\G$ associated with the crossed-module 
\[
\xymatrix{Z^3(G,k^*) & C^2(G,k^*)/B^2(G,k^*) \ar[l]_(.6)\partial}
\]
is equivalent to $\Mex(\Rep(G))$.
\end{lemma}
\begin{proof}
Objects of the categorical group $\G$ are elements of $Z^3(G,k^*)$.  A morphism between $\alpha,\beta\in Z^3(G,k^*)$ is a $2$-cochain $c\in C^2(G,k^*)/B^2(G,k^*)$ such that $\partial(c)=\alpha\cdot\beta^{-1}$.  
Define a functor $F:\G\to\Mex(\Rep(G))$ by $F(\alpha)=\Z(G,\alpha)$.
For a morphism $c:\alpha\to\beta$ define a braided tensor functor $Id(c):\V(G,\alpha)\to \V(G,\beta)$, which is the identity functor with the tensor structure given by the 2-cochain $c$.
Define $F(c):\Z(G,\alpha)\to\Z(G,\beta)$ to be the functor $\Z(Id(c)):\Z(\V(G,\alpha))\to \Z(\V(G,\beta))$ induced by $Id(c)$.  By proposition \ref{p0m} and lemma \ref{p1m}, the effects of $F$ on both $\pi_0$ and $\pi_1$ are bijective. Hence $F$ is an equivalence of categorical groups.
\end{proof}

\begin{remark}
Similarly ${\mathbb M}{\bf ex}(\Rep(G))$ corresponds to the crossed-complex
\[
\xymatrix{Z^3(G,k^*) & C^2(G,k^*) \ar[l]_(.5)\partial & C^1(G,k^*) \ar[l]_(.5)\partial}\ .
\]
\end{remark}

%%%%%%%%%%%%%%%%%%%%%

\subsection{Invertible objects of $\Z(G,\alpha)$}

Here we determine when the category $\Z(G,\alpha)$ is pointed, i.e., when all simple objects of $\Z(G,\alpha)$ are invertible.

Denote 
\[
A(G,\alpha)=\left\{\left.(z,c)\in Z(G)\times C^{1}(G,k^*)\ \right|c(x)c(y) = \alpha(x,y|z)c(xy)\quad \forall x,y\in G\right\}\ .
\]
It follows from identities \eqref{pca}  that the operation 
\begin{equation}\lb{mp}(z,c)(w,d)=\left(zw,cd\alpha(-|z,w)\right)\end{equation}
makes $A(G,\alpha)$ a group. 

\begin{proposition}\label{Picard}
$\Inv(\Z(G,\alpha))\simeq A(G,\alpha)$.
\end{proposition}
\begin{proof}
As a $G$-graded vector space, an invertible object of $\Inv(\Z(G,\alpha))$ has to be one dimensional and concentrated in a single degree $z\in Z(G)$.  The $\alpha$-projective $G$-action \eqref{pa} on such object is given by multiplication by $c\in C^1(G,k^*)$, satisfying $c(x)c(y) = \alpha(x,y|z)c(xy)$. 
The tensor product \eqref{tp} in $\Z(G,\alpha)$ corresponds to \eqref{mp}.
\end{proof}

\begin{remark}
Recall from \cite{gjs} that the group of isomorphism classes of invertible objects $\Inv(\Z(\C))$ of the monoidal centre $\Z(\C)$ fits into an exact sequence
\[
1\longrightarrow\Hom(\Aut_\ot(Id_\C),k^*)\longrightarrow\Inv(\Z(\C))\longrightarrow\Inv(\C)
\]
For the Drinfeld centre $Z(G,\alpha)$, this takes the form
\[
1\longrightarrow\widehat{G}\overset{\iota}{\longrightarrow}A(G,\alpha)\overset{\pi}{\longrightarrow}Z_\alpha(G)\longrightarrow 1\ ,
\]
where $Z_\alpha(G)$ is the following subgroup of the centre of $G$:
\[
Z_{\alpha}(G)=\left\{z\in Z(G)\left|\ \exists c\in C^1(G,k^*)\ \text{such that}\ c(x)c(y)=\alpha(x,y|z)c(xy)\ \forall x,y\in G\right.\right\}\ .
\]
\end{remark}

We call a $3$-cocycle $\alpha\in Z^3(G,k^*)$ {\em soft} if, for any $g\in G$, there is a $c\in C^1(G,k^*)$ such that $c(x)c(y) = \alpha(x,y|g)c(xy)$ for all $x,y\in G$.

A fusion category $\C$ is {\em pointed} if every simple object of $\C$ is invertible. 

\begin{corollary}\lb{zga}
The category $\Z(G,\alpha)$ is pointed if and only if $G$ is abelian and $\alpha$ is soft.
\end{corollary}
\begin{proof}
The category $\Z(G,\alpha)$ is pointed iff its group of invertible objects is of order $|G|^2$.  
From proposition \ref{Picard}, $|\widehat{G}|=|G| = |Z_\alpha(G)|$. Both equalities imply that $G$ is abelian; the second also implies that $\alpha$ is soft.
\end{proof}

Now assume that $G=B$ is abelian.  Denote by $Z^3_{soft}(B,k^*)$ the group of soft 3-cocycles. It is straightforward that 3-coboundaries are soft.  Write $H^3_{soft}(B,k^*)=Z^3_{soft}(B,M)/B^3(B,k^*)$.
\begin{lemma}
Let $B$ be an abelian group.   Then the third soft cohomology group $H^3_{soft}(B,k^*)$ is the kernel of the alternation map ${Alt}_3:H^3(B,k^*)\to\Hom{\left(\Lambda^3B,k^*\right)}$.
\end{lemma}
\begin{proof}
The assignment $\alpha\mapsto \alpha(-,-|-)$ gives a homomorphism
\[
H^3(B,k^*)\to H^2(B,\Map(B,k^*))
\]
into the second cohomology of $B$ with coefficients in the (trivial) $B$-module of set-theoretic maps $\Map(B,k^*)$. 
The soft cohomology $H^3_{soft}(B,k^*)$ is the kernel of this homomorphism.
Since $k^*$ (and $\Map(B,k^*)$) is divisible, the universal coefficient formula implies that the alternation map ${Alt}_2$ gives an isomorphism
\[
H^2(B,\Map(B,k^*))\simeq \Hom{\left(\Lambda^2B,\Map(B,k^*)\right)}\ .
\]
The commutative diagram
\[
\xymatrix{
0\ar[r] & H^{3}_{soft}(B,k^*)\ar[r]&H^{3}(B,k^*)\ar[r]\ar[d]_{{Alt}_3}&H^2{\left(B,\Map(B,k^*)\right)}\ar[d]^{{Alt}_2}\\
& \null&\Hom(\Lambda^3B,k^*)\ \ar@{>->}[r]&\Hom{\left(\Lambda^2B,\Map(B,k^*)\right)}  
}
\]
gives the result.
\end{proof}

For an abelian group $B$, define $\Mex_{{pt}}(\Rep(B))$ to be the full categorical subgroup of $\Mex(\Rep(B))$ consisting of those modular extensions $\D$ of $\Rep(B)$ that are pointed categories.  The computations above give the following:

\begin{theorem}\lb{fmt}
The group of connected components $\pi_0(\Mex_{pt}(B)) = {Mex}_{pt}(\Rep(B))$ coincides with $H^3_{soft}(B,k^*)$ and fits into the short exact sequence
\[
0\ \longrightarrow\ {Mex}_{pt}(\Rep(B))\ \longrightarrow H^3(B,k^*)\ \longrightarrow\ \Hom{\left(\Lambda^3B,k^*\right)}\ \longrightarrow\ 0
\]
\end{theorem}

%%%%%%%%%%%

\section{Lagrangian extensions of abelian groups}

\subsection{Pointed braided fusion categories}\lb{pbfc}

Here we quickly recall a description of pointed braided fusion categories, i.e., braided fusion categories with invertible simple objects.

The isomorphism classes of simple objects of a pointed braided fusion category $\C$ form an abelian group $A = \Inv(\C)$. The category $\C$ is equivalent to the category of $A$-graded vector spaces with associativity constraint and braiding given by 
\[
\alpha_{V,U,W}(v\otimes(u\otimes w))=\alpha(x,y,z)((v\otimes u)\otimes w),\qquad v\in V_x,\ u\in U_y,\ w\in W_z,\ x,y,z\in A,
\]
\begin{equation}\label{braid}
c_{V,W}(v\otimes w)=c(x,y)(w\otimes v),\qquad v\in V_x,\ w\in W_y\ .
\end{equation}
correspondingly.  The unit object conditions mean that $\alpha\in C^3(A,k^*)$ and $c\in C^2(A,k^*)$ are normalised cochains.  The pentagon axiom is equivalent to the $3$-cocycle condition for $\alpha$.  The hexagon axioms amount to 
\begin{align*}c(x,z)c(y,z)\alpha(x,z,y) & = c(x+y,z)\alpha(x,y,z)\alpha(z,x,y)\\ 
c(x,z)c(x,y)\alpha(x,y,z) & = c(x,y+z)\alpha(y,x,z)\alpha(y,z,x)
\end{align*} 
for all $a,b,c\in A$.  
A pair $(\alpha,c)\in Z^3(A,k^*)\times C^2(A,k^*)$ satisfying to the above conditions is called an {\em abelian 3-cocycle}.  We denote the corresponding pointed braided fusion category by $\C(A,\alpha,c)$. 

Structures of braided equivalences on the identity functor correspond to the following coboundary relation on abelian 3-cocycles.  An abelian $3$-cocycle $(\alpha,c)$ is an abelian {\em $3$-coboundary} if $(\alpha,c)=(\d u(x,y,z),u(z,y)u(y,z)^{-1})$ for some $u\in C^2(A,k^*)$.  The {\em third abelian cohomology} is the quotient $H^3_{ab}(A,k^*)=Z^{3}_{ab}(A,k^*)/B^{3}_{ab}(A,k^*)$.

Eilenberg and Mac Lane gave a convenient description of $H^3_{ab}(A,k^*)$.
A function $q:A\rightarrow k^*$ is {\em quadratic} if $q(nx)=q(x)^{n^{2}}$ for all $n\in\bZ$ and $x\in A$, and the function $\sigma:A\times A\to k^*$ defined by $\sigma(x,y)=q(x+y)^{-1}q(x)q(y)$ (the {\em polarisation} of $q$) is bi-multiplicative. Denote by $Q(A,k^*)$ the group of quadratic functions.  
It was shown in \cite{em} (see also \cite{js}) that the assignment $(\alpha,c)\mapsto q$, where $q:A\to k^*$ is defined by $q(x)=c(x,x)$, gives an isomorphism
\[
H^{3}_{ab}(A,k^*)\overset{\sim}{\longrightarrow} Q(A,k^*)\ .
\]
Slightly abusing notation, we will write $\C(A,q)$ instead of $\C(A,\alpha,c)$.  Where the quadratic function $q$ is trivial, we write $\C(A,q)=\C(A)$. The category $\C(A)$ is symmetric and coincides with $\Rep(\widehat A)$.

A quadratic function $q:A\to k^*$ is {\em non-degenerate} if its polarisation $\sigma$ is non-degenerate, i.e.,
\[
\ker(\sigma) = \{x\in A|\ q(x+y) = q(x)q(y)\quad \forall y\in A\} = 0\ .
\]
The pointed category $\C(A,q)$ is non-degenerate as a braided fusion category if and only if $q$ is non-degenerate.

The {\em orthogonal complement} of a subgroup $B$ in a quadratic abelian group $(A,q)$ is  
\[
B^\perp=\{x\in A\ |\ q(x+y)=q(x)q(y)\ \forall y\in B\}\ .
\]  
Note that the symmetric centraliser \eqref{syce} of $\C(B,q|_B)$ in $\C(A,q)$ is $\C(B^\perp,q|_{B^\perp})$. 

A subgroup $B\subset A$ is {\em isotropic} if $q|_B = 0$. If $B$ is an isotropic subgroup of $A$, then $B\subset B^\perp$.  In this case, we call the quotient group $B^\perp/B$ the {\em isotropic contraction} along $B$.  The quadratic function $q$ descends to the isotropic contraction $B^\perp/B$.
An isotropic subgroup $L$ of a non-degenerate quadratic group $(A,q)$ is {\em Lagrangian} if $L=L^\perp$. Note that  $L$ is Lagrangian if and only if $|L|^2=|A|$.
In other words, the pointed category $\C(A,q)$ is a modular extension of the symmetric category $\C(L)$ if and only if $L$ is a Lagrangian subgroup of a non-degenerate quadratic group $(A,q)$.

\begin{remark}\lb{pzg}
The braiding \eqref{br} on $\Z(G,\alpha)$ induces a quadratic function on the group of invertible objects $q:A(G,\alpha)\to k^*$ given by $q(z,c) = c(z)$.  When $B$ is abelian and $\alpha\in Z^3(B,k^*)$ is soft, the group $A(B,\alpha)$ is a Lagrangian extension of $\widehat{B}$ with respect to the quadratic function $q$.
\end{remark}

We finish this preliminary section by recalling basic facts about etale algebras in pointed braided fusion categories (see, e.g., \cite{dmno} for details).

The graded support of an indecomposable separable algebra in a pointed fusion category is a subgroup of the grading group.  Etale algebras in the pointed category $\C(A,q)$ correspond to isotropic subgroups of $A$. The category of modules $\C(A,q)_R$ over the etale algebra $R=R(B)$ corresponding to an isotropic $B\subset A$ is pointed with the group of invertible objects $A/B$. The category of local modules $\C(A,q)_{R(B)}^{loc}$ is braided equivalent to $\C(B^\perp/B,\overline q)$. In particular, Lagrangian algebras in $\C(A,q)$ correspond to Lagrangian subgroups of $A$.

\subsection{Lagrangian extensions}\lb{lagext}

Let $B$ be an abelian group.  A {\em Lagrangian extension of $B$} is a triple $(A,q,\iota)$, where $A$ is an abelian group, $\iota:B\hookrightarrow A$ is an embedding of groups, and $q:A\to k^*$ is a non-degenerate quadratic function such that $B\subset A$ is a Lagrangian subgroup with respect to $q$. Lagrangian extensions of $B$ form a category $\Lex(B)$.  
A morphism $(A,q,\iota)\to(A',q',\iota')$ in $\Lex(B)$ is a homomorphism $\omega:A\to A'$ such that $q'\circ\omega=q$ and $\omega\circ\iota=\iota'$.
Note that all morphisms in $\Lex(B)$ are invertible, i.e. $\Lex(B)$ is a groupoid.

The results of the previous section imply the following.
\begin{proposition}
The functor 
\begin{equation}\lb{flm}
\Lex(B)\ \to\ \Mex_{pt}(\C(B)),\qquad (A,q,\iota)\mapsto\C(A,q)\ .
\end{equation} 
is an equivalence of groupoids.
\end{proposition}
\begin{proof}
Indeed, any pointed modular extension of $\C(B)$ has a form $\C(A,q)$ for some Lagrangian extension $\C(A,q)$ of $B$.
Any equivalence of extensions $\C(A,q)\to \C(A',q')$ corresponds to an isomorphism of Lagrangian extensions $(A,q)\to (A',q')$.
\end{proof}

The equivalences \eqref{flm} allows us to transfer the monoidal structure $\odot_{\C({B})}$ from $\Mex_{pt}(\C(B))$ to $\Lex(B)$.
Note first that the etale algebra $R=F(I)\in \C(B)\boxtimes\C(B)$ corresponds under identification $\C(B)\boxtimes\C(B)\simeq \C(B\op B)$ to the anti-diagonal subgroup $\bar{\delta}(B)=\{(b,-b)|\ b\in B\}\subset B\op B$.
Indeed,the tensor product functor $C(B)\boxtimes\C(B)\to\C(B)$ corresponds to the direct image of the addition homomorphism $a:B\times B\to B,\ (x,y)\mapsto x+y$. Thus $R$ is the inverse image $a^*(I) = R(\bar{\delta}(B))$.  

Now the chain of braided equivalences
\[
\C(A,q)\odot_{\C({B})}\C(A',q') = (\C(A,q)\boxtimes\C(A',q'))^{loc}_R\simeq \C(A\times A',q\times q')^{{loc}}_{R(\overline{\delta}(B))}\simeq
\]
\[
\simeq\C{\left(\overline{\delta}(B)^\perp/\overline{\delta}(B),\overline{q\times q'}\right)}
\]
identifies  $\C(A,q)\odot_{\C({B})}\C(A',q')$ with the pointed category corresponding to the isotropic contraction of $\bar{\delta}(B)$ in the quadratic group $(A\times A',q\times q')$. We denote this isotropic contraction by $(A,q,\iota)\boxplus(A',q',\iota')$ and call it the {\em sum} of Lagrangian extensions $(A,q,\iota)$ and $(A',q',\iota')$. Note also that the embedding $B\subset \overline{\delta}(B)^\perp/\overline{\delta}(B)$ given by $b\mapsto (b,0)$ makes the sum a Lagrangian extension of $B$. 
\begin{remark}\lb{orco}
Here is a more explicit description of the orthogonal complement:
\[
\overline{\delta}(B)^\perp = \{(a,a')\in A\times A'|\ q(a+\iota(x))q'(a'-\iota'(x))=q(a)q'(a')\quad\forall x\in B\}\ .
\]
Note that it also coincides with the fibered product $A\times_{\widehat{B}}A'$ with respect to the natural surjections $A\to \widehat{B} \leftarrow A'$.
\end{remark}

By the {\em trivial} Lagrangian extension, we mean $(B\times\widehat{B}, q_\mathrm{std}, \iota)$, with $q_\mathrm{std}(b,\chi) = \chi(b)$ and $\iota(b) = (b,0)$. The trivial Lagrangian extension is the quadratic group corresponding to the monoidal centre $\Z(\C(B))$ and is the monoidal unit object in $\Lex(B)$. 
Define the {\em conjugate} of the Lagrangian extension $(A,q,\iota)$ by $\overline{(A,q,\iota)} = (A,-q,\iota)$.
The conjugate is the dual object in $\Lex(B)$.

We denote by $Lex(B) = \pi_0(\Lex(B))$ the {\em group of Lagrangian extensions}.
The equivalence \eqref{flm} implies that $Lex(B)$ is isomorphic to $Mex_{pt}(\C(B))$.
Using that $\C(B)\simeq \Rep(\widehat B)$ as symmetric fusion categories, we have $Lex(B) = Mex_{pt}(\Rep(\widehat B))$.
Theorem \ref{fmt} now gives the following.
\begin{corollary}\lb{lexzn}
The group of Lagrangian extensions fits into the short exact sequence
\[
0\ \longrightarrow\ Lex(\widehat B)\ \longrightarrow H^3(B,k^*)\ \longrightarrow\ \Hom(\Lambda^3B,k^*)\ \longrightarrow\ 0
\]
\end{corollary}
Corollary \ref{lexzn} can be used to describe the third cohomology $H^3(B,k^*)$. 
But here we are going to use it in the opposite way to get some basic information about groups of Lagrangian extensions.

\begin{remark}
The third cohomology group $H^3(\bZ/n\bZ,k^*)$ is cyclic of order $n$.  Since $\Lambda^3(\bZ/n\bZ)=0$, corollary \ref{lexzn}
gives $Lex(\bZ/n\bZ)\simeq\bZ/n\bZ$.
\end{remark}

In what follows we replace $k^*$ by $\bQ/\bZ$ to simplify the exposition.

\begin{example}\lb{lexz2}
Here, we describe explicitly Lagrangian extensions from $Lex(\bZ/2\bZ)$.  
Since a Lagrangian extension $A\supset \bZ/2\bZ$ is an abelian group extension of $\bZ/2\bZ$ by $\bZ/2\bZ$, $A$ is of order 4, and is isomorphic to either $(\bZ/2\bZ)^2$ or $\bZ/4\bZ$.

Let $A=(\bZ/2\bZ)^2$.  The matrices of possible non-degenerate symmetric bilinear forms $\sigma:A\times A\to \bQ/\bZ$ such that the first summand $\bZ/2\bZ$ is Lagrangian are 
\[
\left(\begin{array}{cc} \ 0\ & \ \onehalf\  \\ [12pt] \frac{1}{2}&\frac{1}{2}\\\end{array}\right)\quad \text{and}\quad \left(\begin{array}{ccc}0&&\frac{1}{2}\\ [12pt] \frac{1}{2}&&0\\\end{array}\right)\ .
\]
Here entries are in $\bQ/\bZ$.  For the first one, a compatible quadratic function $q_\pm:A\to\bQ/\bZ$ has the form 
\[
q_\pm(e_1)=0,\ q_\pm(e_2)=\pm\frac{1}{4},\  q_\pm(e_1+e_2)=\mp\frac{1}{4}\ .
\]
The corresponding classes in $Lex(\bZ/2\bZ)$ are isomorphic and nontrivial.  The second bilinear form corresponds to the trivial Lagrangian extension.  Since all classes in $Lex(\bZ/2\bZ)$ are represented by Lagrangian extensions $A\supset \bZ/2\bZ$ with $A= (\bZ/2\bZ)^2$, the second possibility, $A=\bZ/4\bZ$, does not realize.
\end{example}

\begin{example}\lb{lexz}
More generally, let $B=\bZ/2^\ell\bZ$ for some positive integer $\ell$.  As in example \ref{lexz2} above, a Lagrangian extension $A\supset\bZ/2^\ell\bZ$ is an abelian group extension of $\bZ/2^\ell\bZ$ by $\bZ/2^\ell\bZ$, whence $|A|=2^{2\ell}$.  If $A=\bZ/2^{2\ell}\bZ$, then the only possible Lagrangian subgroup of $A$ is $\langle 2\rangle$.  But a straightforward computation shows that if $q:A\to\bQ/\bZ$ is a quadratic function with respect to which the subgroup $\langle 2\rangle$ is isotropic, then the associated bilinear form $\sigma$ is degenerate.  Thus $\bZ/2^{2\ell}\bZ$ is not realizable as a Lagrangian extension of $\bZ/2^\ell\bZ$.  In particular, this shows that $\bZ/4\bZ$ is not realizable as a Lagrangian extension of $\bZ/2\bZ$ (cf. example \ref{lexz2} above).
\end{example}

\begin{example}\lb{eleven}
Let now $A=\bZ/2^{2\ell-1}\bZ\times\bZ/2\bZ$. Consider a non-degenerate quadratic function $q:A\to\bQ/\bZ$  given by 
\[
q(a,b)=\frac{a^2}{2^{2\ell}}-\frac{b^2}{4}\ .
\]
The subgroup $L\subset A$ generated by the pair $\left(2^{\ell-1},1\right)$ is Lagrangian with respect to $q$ and is isomorphic to $\bZ/2^\ell\bZ$.
\end{example}

We finish this section with a quick discussion of the functoriality property of $\Lex$.  Namely, for a homomorphism of abelian groups $f:B\to D$, we will define a functor between symmetric categorical groups $\Lex(B)\to \Lex(D)$.

For a Lagrangian extension $(A,q,\iota)$ of $B$, set $A'  = \frac{(A\times_{\widehat{B}}\widehat{D})\times D}{\overline{\Delta}(B)}$, where $\overline{\Delta}(B)$ is the antidiagonal subgroup $\{(b,-b)|b\in B\}$.
Now define a Lagrangian extension of $D$
\begin{equation}\lb{lexf}
\Lex(f)(A,q,\iota)=\left(\frac{(A\times_{\widehat{B}}\widehat{D})\times D}{\overline{\Delta}(B)},q',\iota'\right)\ ,
\end{equation}
where $q':A'\to\bQ/\bZ$ is given by $q'(a,\chi,d)=q(a)+\chi(d)$. 

\begin{remark}\lb{fsh}
Suppose now that $f$ is surjective, and let $K=\ker(f)$. One can check that in this case $A\times_{\widehat{B}}\widehat{D}$ coincides with the orthogonal complement $K^\perp$ of $K$ in $A$. Moreover, the quotient $\frac{(A\times_{\widehat{B}}\widehat{D})\times D}{\overline{\Delta}(B)}$ coincides with the isotropic contraction $K^\perp/K$ along $K$.  This contains $B/K$, which is isomorphic to $D$. Thus \eqref{lexf} reduces to 
\[
\Lex(f)(A,q,\iota)=\left(K^\perp/K,q|_{K^\perp},\iota'\right)\ ,
\]
where $\iota':D\simeq B/K\to K^\perp/K$ is induced by $\iota:B\to K^\perp$. 
\end{remark}

%%%%%%%%%%%%%%%%%%%%%

\subsection{The categorical group $\mathcal{E}xt(B,D)$}\lb{cge}

Let $B$ and $D$ be abelian groups.  Denote by $\mathcal{E}xt(B,D)$ the category whose objects are abelian group extensions $D\overset{\iota}{\longrightarrow}A\overset{\pi}{\longrightarrow}B$ and whose morphisms are group homomorphisms $\xi:A\to A'$ making the diagram
\[
\xymatrix{
&&A\ar[dr]^{\pi}\ar[dd]^{\xi}&&\\
0\ar[r]&D\ar[ur]^{\iota}\ar[dr]_{\iota'}&&B\ar[r]&0\\
&&A'\ar[ur]_{\pi'}&&\\
}
\]
commute.  Note that such $\xi$ is necessarily an isomorphism, i.e., $\Ex(B,D)$ is a groupoid.

The category $\Ex(B,D)$ is monoidal with respect to the Baer sum.  Recall that the {\em Baer sum} of extensions is
\begin{equation}\lb{baer}
\left(D\overset{\iota}{\longrightarrow}A\overset{\pi}{\longrightarrow}B\right)\boxplus\left(D\overset{\iota'}{\longrightarrow}A'\overset{\pi'}{\longrightarrow}B\right)=\left(D\overset{i}{\longrightarrow}A\times_{B}A'/\bar{\delta}(D)\overset{p}{\longrightarrow}B\right)\ ,
\end{equation}
where $A\times_{B}A'$ is the fibered product, $\bar{\delta}(D)$ is the antidiagonal subgroup, $i=\iota=\iota'$ and $p=\pi=\pi'$.  

The unit object in $\Ex(B,D)$ is the trivial extension $D\to D\oplus B\to B$ with the canonical embedding and projection maps.  The {\em opposite extension} to $D\overset{\iota}{\longrightarrow}A\overset{\pi}{\longrightarrow}B$ is the extension $D\overset{\iota}{\longrightarrow}A\overset{-\pi}{\longrightarrow}B$.  Altogether, this makes the category $\mathcal{E}xt(B,D)$ into a categorical group with respect to $\boxplus$.  The standard invariants of the categorical group $\Ex(B,D)$ are $\pi_0(\Ex(B,D)) = \Ext(B,D)$ and $\pi_1(\Ex(B,D)) = \Hom(B,D)$.

The functor $\Ex$ is functorial in each of its arguments; it is contravariant in its first argument and covariant in its second argument.

In the special case when the arguments $B$ and $D$ are dual to each other, the categorical group $\mathcal{E}xt(B,D)$ has a natural symmetry.  Namely, define the contravariant functor $T:\Ex(\widehat{B},B)\to\Ex(\widehat{B},B)$ by 
\[
T{\left(\xymatrix{B\ar[r]^{\iota}&A\ar[r]^{\pi}&\widehat{B}\\}\right)}=\left(\xymatrix{B\ar[r]^{\widehat{\pi}e}&\widehat{A}\ar[r]^{\widehat{\iota}}&\widehat{B}}\right)\ ,
\]
where $e:B\to \widehat{\widehat{B}}$ is the canonical evaluation isomorphism.  The functor $T$ is an involutive autoequivalence. More precisely, we have a natural isomorphism $i:\Id\to T^2$ (given by $e:A\to {\widehat{\widehat{A}}}$) such that $i_{T(X)} = T(i_X)^{-1}$. 
Namely, let $X\in\Ex(\widehat{B},B)$ be the Lagrangian extension $\xymatrix{B\ar[r]^{\iota}&A\ar[r]^{\pi}&\widehat{B}\\}$, so that $T(X)$ is the extension
\[
\xymatrix{\widehat{\widehat{B}}\ar[r]^{\widehat{\pi}}&\widehat{A}\ar[r]^{\widehat{\iota}}&\widehat{B} }\ .
\]
Define a map $X\to T^2(X)$ by the following diagram:  
\[
\xymatrix{B\ar[r]^\iota\ar[d]_{e_L}&A\ar[r]^\pi\ar[d]^{e_A}&\widehat{B}\\
\widehat{\widehat{B}}\ar[r]_{\widehat{\widehat{\iota}}}&\widehat{\widehat{A}}\ar[r]_{\widehat{\widehat{\pi}}}&\widehat{\widehat{\widehat{B}}}\ar[u]_{\widehat{e_B}}\\}\ .
\]
The left-hand square of this diagram commutes due to naturality of the evaluation map $e:\Id\to\widehat{\widehat{(\ )}}$.  
Note that $e_{\widehat{B}}=\widehat{e_B}^{-1}$.  Indeed, $e_{\widehat{B}}:\chi\mapsto\left(e(a)\mapsto e(a)(\chi)=\chi(a)\right)$ for $a\in B$, whence $\widehat{e_B}\circ e_{\widehat{B}}:\chi\mapsto\left(a\mapsto\chi(a)\right)$ for $a\in B$.  This implies that the right-hand square of the diagram also commutes, and therefore the extensions $X$ and $T^2(X)$ are equivalent.

More generally, let $\G$ be a categorical group.  
Let $T:\G\to \G$ be a contravariant monoidal autoequivalence with a natural isomorphism $i:\Id\to T^2$ such that $i_{T(X)} = T(i_X)^{-1}$ for any $X\in\G$.  We say that an object $X\in\G$ is {\em $T$-equivariant} if there is an isomorphism $x:X\to T(X)$ making the following diagram commute:
\begin{equation}\lb{eco}
\xymatrix{
X\ar[dr]_{x}\ar[rr]^{i_X}&&T^2(X)\ar[dl]^{T(x)}\\
&T(X)\\
}
\end{equation}
Let $\tau$ be the effect of $T$ on $\pi_0(\G)$, and denote by $\pi_0(\G)_0^\tau$ the subgroup of classes of $T$-equivariant objects.

\begin{proposition}
The sequence 
\[
\xymatrix{0 \ar[r] & \pi_0(\G)_0^\tau\ar[r] & \pi_0(\G)^\tau \ar[r] & H^1(\bZ/2\bZ, \pi_1(\G))}
\]
is exact.
\end{proposition}
\begin{proof}
The class of an object $X\in\G$ is {\em $\tau$-invariant} if there is a morphism $x:X\to T(X)$. Define $a = a(x)\in\Aut_\G(X)$ as the counterclockwise composition of the diagram \eqref{eco}.  Let $\alpha(x)\in\Aut_\G(I) =\pi_1(\G)$ be the corresponding (in the sense of remark \ref{acg}) automorphism of the monoidal unit $I$.  Then we have $\tau(\alpha)=T(\alpha)=T(a)$. 
Applying the functor $T$ to the diagram \eqref{eco} gives 
\begin{equation}\lb{eco2}
\xymatrix{
T(X)&&T^3(X)\ar[ll]_{T{\left(i_X\right)}}\\
&T^2(X)\ar[ul]^{T(x)}\ar[ur]_{T^2(x)}\\
}
\end{equation}

Note that $T(a)\in\Aut_\G(T(X))$ is the clockwise composition of the diagram \eqref{eco2}.  We claim that 
\begin{equation}\lb{taai}T(a)=xa^{-1}x^{-1}\end{equation}
In terms of elements of $\Aut_\G(I)$ we have $\tau(\alpha) = \alpha^{-1}$.

Consider the diagram
\begin{equation}\lb{taai1}
\xymatrix
{
X\ar[rrr]^{x}\ar[ddr]_{i_X}&&&T(X)\ar@/^10pt/[dl]^{i_{T(X)}}\ar[ddd]^{T(a)}\\
&&T^3(X)\ar@/^10pt/[ur]^(.3){T(i_X)}&\\
&T^2(X)\ar[ur]_{T^2(x)}\ar[drr]_{T(x)}&&\\
X\ar[uuu]^{a}\ar[rrr]_{x}&&&T(X)\\
}
\end{equation}
The rightmost cell of \eqref{taai1} commutes by definition of $T(a)$.  The bottom left cell commutes by definition of $a$.  The top centre cell commutes by naturality of $i$, and so the diagram commutes overall.  \eqref{taai} follows from commutativity of \eqref{taai1}.
Now let $X\in\G$ be $\tau$-invariant with an isomorphism $x:X\to T(X)$. Let $x':X\to T(X)$ be another isomorphism. Define $b\in\Aut(X)$  by the diagram
\[
\xymatrix{
X\ar[rr]^{x}&&T(X)\\
&X\ar[ul]^{b}\ar[ur]_{x'}\\
}
\]

It follows from the  diagram 
\[
\xymatrix
{
X\ar[rrrr]^{a(x')}\ar[dr]_{b}\ar[ddd]_{x'}&&&&X\ar[ddd]^{i_X}\\
&X\ar[ddl]^{x}&&X\ar[ll]_{xT(b)x^{-1}}\ar[ur]_{a(x)}\ar[dl]^{x}&\\
&&T(X)\ar[dll]^{T(b)}&&\\
T(X)&&&&T^2(X)\ar[ull]^{T(x)}\ar[llll]^{T(x')}\\
}
\]
that $a(x') = a(x)xT(b)^{-1}x^{-1}b$.
This translates to $\alpha(x') = \alpha(x)\tau(\beta)^{-1}\beta$, where $\beta\in\Aut_\G(I)$ is the automorphism corresponding to $b\in \Aut_\G(X)$. 

Thus we have a map 
\[
\pi_0(\G)^\tau \to H^1(\bZ/2\bZ,\pi_1(\G))\ ,
\]
sending $X\in\pi_0(\G)^\tau$ to the class of $\alpha(x)$ in
\[
H^1(\bZ/2\bZ,\pi_1(\G))= \{\alpha\in\pi_1(\G)\ |\tau(\alpha)=\alpha^{-1}\}/\{\beta\tau(\beta)^{-1}\ |\ \beta\in\pi_1(G)\}\ .
\]

Now we check that this map is a homomorphism.
Given $X,Y\in\pi_0(\G)^\tau$, choose $x:X\to T(X)$ and $y:Y\to T(Y)$.  Then 
\[
\xymatrix{X\ot Y\ar[r]^(.4){x\ot y}&T(X)\ot T(Y)\ar[r]^{T_{X,Y}}&T(X\ot Y)}
\]
is a weak equivariance structure for $X\otimes Y$.  Compute $a(x\ot y)=a(x|y)$ according to the following diagram:
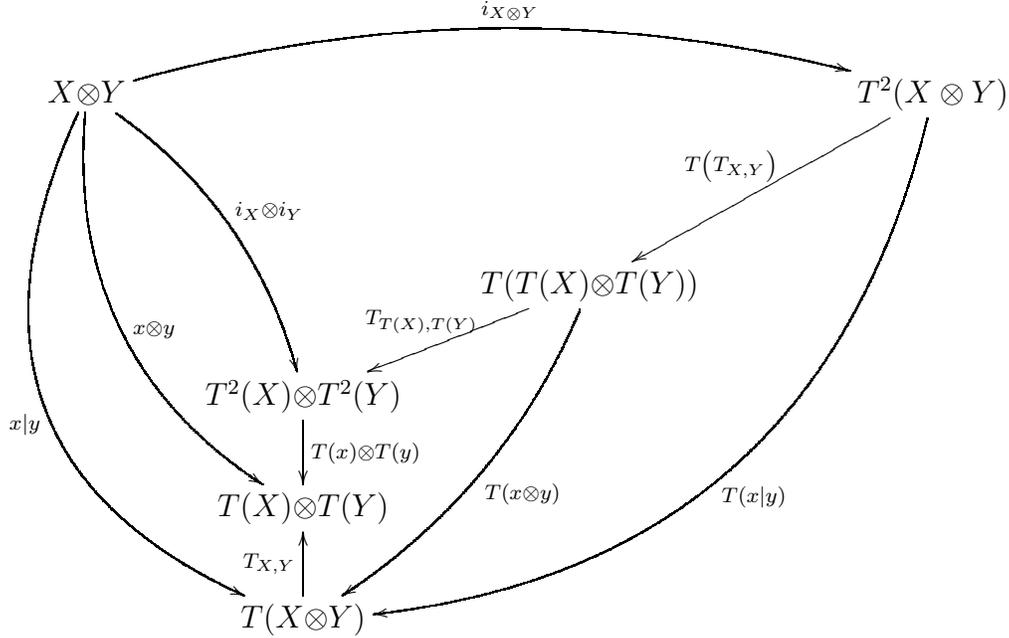
\[
\xymatrix
{
X\ot Y\ar@/^15pt/[dddr]^{i_{X}\ot i_{Y}}\ar@/^25pt/[rrrr]^{i_{X\ot Y}}\ar@/_30pt/[ddddr]^{x\ot y}\ar@/_60pt/[dddddr]_{x|y}&&&&T^2{\left(X\otimes Y\right)}\ar[ddll]_{T{\left(T_{X,Y}\right)}}\ar@/^60pt/[dddddlll]^{T(x|y)}\\
&&&&\\
&&T{\left(T(X)\ot T(Y)\right)}\ar[dl]_{T_{T(X),T(Y)}}\ar@/^15pt/[dddl]^{T(x\ot y)}&\\
&T^2(X)\ot T^2(Y)\ar[d]^{T(x)\ot T(y)}&&\\
&T(X)\ot T(Y)&&\\
&T{\left(X\ot Y\right)}\ar[u]^{T_{X,Y}}&&\\
}
\]
It follows from the above diagram that $a(x|y)=a(x)\ot a(y)$.  Finally, by the construction, the class of $\alpha(X)$ is trivial iff a $\tau$-invariant structure on $X$ can be promoted to a $T$-equivariant one.
\end{proof}

The effect of $T$ on $\pi_0(\Ex(\widehat{B},B))=\Ext(\widehat{B},B)$ is an involution $\tau=\pi_0(T):\Ext(\widehat{B},B)\to\Ext(\widehat{B},B)$, which sends an extension of $\widehat{B}$ by $B$ to its dual.  By factoring through the natural identification of the double dual with $B$, the dual extension becomes another extension of $\widehat{B}$ by $B$.

If we identify $\Hom(\widehat{B},B)$ with the tensor square $B^{\ot2}$,  the effect of $T$ on $\pi_1(\Ex(\widehat{B},B))$ becomes the permutation involution.

\begin{remark}\lb{h1}
The above proposition gives an exact sequence
\[
\xymatrix{0 \ar[r] & \Ext(\widehat{B},B)_0^\tau\ar[r] & \Ext(\widehat{B},B)^\tau \ar[r] & H^1(\bZ/2\bZ, B^{\ot2})}
\]
According to the example \ref{additive},   $H^1(\bZ/2\bZ, B^{\ot2})\simeq B_2$ functorially.  
This gives an exact sequence
\[
\xymatrix{0 \ar[r] & \Ext(\widehat{B},B)_0^\tau\ar[r] & \Ext(\widehat{B},B)^\tau \ar[r] & B_2}
\]

\end{remark}

%%%%%%%%%%%%%%%%%%%%%

\subsection{Comparing $\Lex(B)$ with $\Ex(\widehat{B},B)$}\lb{cle}

Define a functor $F:\Lex(B)\to\mathcal{E}xt(\widehat{B},B)$ by 
\[
F(A,q,\iota)=\left(B\overset{\iota}{\longrightarrow}A\overset{\pi}{\longrightarrow}\widehat{B}\right),
\]
where $\left(\pi(a)\right)(x)=q(a+\iota(x))-q(a)$ for $a\in A,x\in B$.  
Note that $F$ is monoidal with the obvious monoidal structure.

The functor $F$ induces a homomorphism of abelian groups $\phi:Lex(B) \to\Ext(\widehat{B},B)$.  
First we examine the image of $\phi$. 
\begin{lemma}\lb{teq}
The extension $F(A,q,\iota)$ is $T$-equivariant.
\end{lemma}
\begin{proof}
Define $f:A\to\widehat{A}$ by $f(a)=\sigma(a,-)$ for $a\in A$.  Here, $\sigma$ is the polarisation of $q$.  Such $f$ makes the diagram
\[
\xymatrix{
B\ar[r]^{\iota}\ar[d]_{e}&A\ar[r]^{\pi}\ar[d]^f &\widehat{B}\ar@{=}[d]\\
\widehat{\widehat{B}}\ar[r]_{\widehat{\pi}}&\widehat{A}\ar[r]_{\widehat{\iota}}&\widehat{B} 
}
\]
commute. Indeed, for $x\in B$ and $a\in A$, we have $(f(\iota(x)))(a)=\sigma{\left(\iota(x),a\right)}=(\pi(a))(x)=e(x)(\pi(a))=\widehat{\pi}(e(x))(a)$, and $(\pi(a))(x)=\sigma(a,x)=(f(a))(x)=(\widehat{\iota}(f(a)))(x)$.
It is straightforward to check that $f$ is $T$-equivariant structure on $F(A,q,\iota)$. 
\end{proof}

Lemma \ref{teq} shows that the homomorphism $\phi$ lands in $\Ext(\widehat{B},B)^\tau$, with the image of $\phi$ being precisely $\Ext(\widehat{B},B)_0^\tau$.
Denote by $K(B)$ and $C(B)$ the kernel and the cokernel, respectively, of the homomorphism $\phi:Lex(B)\to\Ext(\widehat{B},B)^\tau$. Clearly, $K$ and $C$ are functors from the category of finite abelian groups to itself that fit into the exact sequence
\[
\xymatrix{0\ar[r] & K(B)\ar[r] & Lex(B)\ar[r]^(.4)\phi & \Ext(\widehat{B},B)^\tau\ar[r] & C(B)\ar[r] & 0}
\]

As the first step in describing the functors $K$ and $C$, we compute the polarisation (see Appendix \ref{pola}) of the functor $Lex$.  Note that for an extension 
\[
\xymatrix{
0\ar[r]&B\ar[r]^{\iota}&E\ar[r]^\pi&\widehat{D}\ar[r]&0
}
\]
the embedding $B\oplus D\hookrightarrow E\oplus\widehat{E}$ (the direct sum of $\iota$ and $\widehat{\pi}\circ e^{-1}$) is a Lagrangian subgroup of the quadratic group $(E\oplus\widehat{E},q_\mathrm{std})$.  In other words, $(E\oplus\widehat{E},q_\mathrm{std})$ becomes a Lagrangian extension of $B\op D$.  This defines a map $\upsilon:\Ext(\widehat{D},B)\to Lex(B\oplus D)$.

\begin{proposition}\lb{plx}
The sequence
\[
0\longrightarrow\Ext(\widehat{D},B)\overset{\upsilon}{\longrightarrow}Lex(B\oplus D)\overset{\kappa}{\longrightarrow}Lex(B)\oplus Lex(D)\longrightarrow 0\ ,
\]
where $\kappa:Lex(B\oplus D)\to Lex(B)\oplus Lex(D)$ is the canonical mapping, is  exact.
\end{proposition}
\begin{proof}
By remark \ref{fsh}, $\kappa$ maps the class of a Lagrangian extension $(A,q,\iota)$ (of $B\oplus D$) into the pair of Lagrangian extensions
\[
B \subset D^\perp/D,\qquad  D\subset B^\perp/B\ .
\]
For $(A,q,\iota)$  in the kernel of $\kappa$, we have 
\[
B\subset D^\perp/D\ \simeq\ B\oplus\widehat{B}, \qquad D\subset B^\perp/B\ \simeq\  D\oplus\widehat{D}\ ,
\]
where $B$ (respectively $D$) is embedded as the first summand and is Lagrangian with respect to the standard quadratic function on $B\oplus\widehat{B}$ (respectively $D\oplus\widehat{D}$).
Now we want to show that the Lagrangian extension $(E\oplus\widehat{E},q_\mathrm{std})$ of $B\op D$ (the effect of the map $\upsilon$ on an extension $B\to E\to \widehat{D}$) is always in the kernel of $\kappa$.  Indeed, the orthogonal complement of $D$ in $E\oplus\widehat{E}$ has the form
\[
D^\perp=\left\{\left.\left(e,\varepsilon\right)\in E\oplus\widehat{E}\ \right|\sigma_\mathrm{std}((e,\varepsilon),(0,\widehat{\pi}(d))=0\ \forall d\in D\right\}\ =\ B\oplus\widehat{E}\ .
\]
Furthermore,
\[
D^\perp/D\simeq (B\oplus\widehat{E})/{\left(0\oplus\widehat{\pi}(D)\right)}\simeq B\oplus\widehat{B}\ .
\]
Analogous results hold for $B^\perp/B$.  This shows that $\mathrm{im}(\upsilon)\subseteq\ker(\kappa)$.
To see that $\mathrm{im}(\upsilon)=\ker(\kappa)$, we show that, for $(A,q,\iota)\in\ker(\kappa)$, there exists an extension $B\overset{\iota}{\longrightarrow}E\overset{\pi}{\longrightarrow}\widehat{D}$ such that
\[
(A,q,\iota)\simeq\left(E\oplus\widehat{E},q_\mathrm{std},\iota\oplus\widehat{\pi}\right)\ .
\]
We start by extracting such $E$ out of $A$.  Since $(A,q,\iota)$ is in the kernel of $\kappa$, the isotropic contraction along $B$ is trivial as a Lagrangian extension of $D$:   $B^\perp/B\simeq D\oplus\widehat{D}$.  By lifting the corresponding projections $p_1:B^\perp/B\twoheadrightarrow D$, $p_2:B^\perp/B\twoheadrightarrow\widehat{D}$, we obtain surjections $\widetilde{p}_1:B^\perp\twoheadrightarrow D$ and  $\widetilde{p}_2:B^\perp\twoheadrightarrow\widehat{D}$.  Now define $E$ to be $\ker(\widetilde{p}_1)$.  Since $\ker(\widetilde{p}_1)\cap\ker(\widetilde{p}_2)=B$, the sequence
\begin{equation}\lb{ee}
B\overset{\iota}{\longrightarrow}E\overset{\widetilde{p}_2}{\longrightarrow}\widehat{D}
\end{equation}
is short exact.  Similarly, $D^\perp/D\simeq B\oplus\widehat{B}$, and so we have projections $w_1:D^\perp/D\twoheadrightarrow B$, $w_2:D^\perp/D\twoheadrightarrow\widehat{B}$ and corresponding surjections $\widetilde{w}_1:D^\perp\twoheadrightarrow B$, $\widetilde{w}_2:D^\perp\twoheadrightarrow\widehat{B}$.  
Define $E'=\ker(\widetilde{w}_1)$.  We then have the analogous short exact sequence
\begin{equation}\lb{ee'}
D\overset{\iota'}{\longrightarrow}E'\overset{\widetilde{w}_2}{\longrightarrow}\widehat{B}\ .
\end{equation}

By their definitions, $E$ and $E'$ are subgroups of $A$.  We show that $A$ is their direct sum. Since the orders of $E$ and $E'$ square to the order of $A$, all we need to show is that $E\cap E'=0$.  Let $x\in E\cap E'$.  Then $x\in B^\perp\cap D^\perp=(B\oplus D)^\perp=B\oplus D$ since $B\oplus D$ is Lagrangian in $A$.  Since $\widetilde{p}_1|_{B\oplus D}$ is the second projection and $\widetilde{w}_1|_{B\oplus D}$ is the first, $\widetilde{p}_1(x)=\widetilde{w}_1(x)=0$ implies that $x=0$.  

Now we show that $E$ and $E'$ are Lagrangian subgroups of the quadratic group $A$.  Together with the decomposition $E\op E'=A$, that will give us a non-degenerate pairing between $E$ and $E'$, identifying $E'$ with $\widehat E$.
What we are actually going to show is that $E$ is isotropic (this is sufficient, since $|E|^2=|A|$).  Recall that the kernel of the restriction of $q$ to $B^\perp$ is $B$, so for any $x\in B^\perp$ we have $q(x)=q_{B^\perp/B}(\bar{x})$, where $\bar{x}$ is the coset of $x$ modulo $B$.  Now for $x\in E=\ker(\widetilde{p}_1)$, the coset $\bar{x}$ lies in $\ker(p_1)$, which is an isotropic subgroup of $B^\perp/B$. Hence $q(x)=q_{B^\perp/B}(\bar{x})=0$.   Similarly for $E'$.
The last thing we need to check is that, upon identification between $E'$ and $\widehat E$, the extension \eqref{ee'} is the dual of \eqref{ee}.  In other words, we need $\widehat{\iota}=\widetilde{w}_2$ and $\widehat{\iota'}=\widetilde{p}_2$, which follows immediately from the definitions.
\end{proof}

\begin{corollary}
The functors $K$ and $C$ are additive.
\end{corollary}
\begin{proof}
The polarisation of the functor $B\mapsto \Ext(\widehat{B},B)^\tau$ is $\Ext(\widehat{D},B)$.  Moreover, the natural transformation $\phi$ induces an isomorphism of polarisations.  Finally, note that  $K(0) = C(0) = 0$.
\end{proof}

First we look at the kernel functor $K$.  
\begin{lemma}
The functor $K$ fits into the exact sequence
\begin{equation}\lb{sequence}
Hom_\bZ(\widehat B\ot \widehat B,\bQ/\bZ)\to Q(\widehat B,\bQ/\bZ)\to K(B)\to 0
\end{equation}
\end{lemma}
\begin{proof}
Let $(A,q,\iota)\in K(B)$.  Then, as an abelian group, $A$ can be identified with $B\times\widehat{B}$.  Moreover, under this identification, $\iota$ becomes the canonical inclusion. Furthermore, the quadratic function $q$ is such that 
$q(b+b',\chi)-q(b',\chi) = \pi(b',\chi)(b)=\chi(b)$.  Denote by $\tilde{q}$ the restriction of $q$ to $\widehat{B}$.  Then for $b\in B,\chi\in\widehat{B}$, we have $q(b,\chi)=\chi(b)+\tilde{q}(\chi)$  Noting that $\chi(b)=q_\mathrm{std}(b,\chi)$, one can write $(A,q,\iota)\simeq (B\times\widehat{B},q_\mathrm{std}+\tilde{q},\iota)$ as Lagrangian extensions of $B$.  In other words, the map $Q(\widehat{B},\bQ/\bZ)\to K(B)$ given by $\tilde{q}\mapsto(B\times\widehat{B},q_\mathrm{std}+\tilde{q},\iota)$ is an epimorphism of groups.
Note that the kernel of the map $Q(\widehat{B},\bQ/\bZ)\to K(B)$ is the image of the map $Hom_\bZ(\widehat B\ot \widehat B,\bQ/\bZ)\to Q(\widehat{B},\bQ/\bZ)$ sending a character $\xi$ on $\widehat B\ot \widehat B$  to the quadratic function $q(\chi) = \xi(\chi\ot\chi))$.  
\end{proof}

According to the example \ref{J} of the appendix, $K(B)=J(\widehat B)$.
\void{
For $B$ of odd order, $K(B)=0$.  For $B=\bZ/2^\ell\bZ$, we have $Q(\widehat B,\bQ/\bZ)\simeq\bZ/2^{\ell+1}\bZ$ and $Hom_\bZ(\widehat B\ot \widehat B,\bQ/\bZ)\simeq\bZ/2^\ell\bZ$,  so the sequence \eqref{sequence} forces $K(B)\simeq\bZ/2\bZ$.  The effect on morphisms is as follows:  on isomorphisms between cyclic 2-groups, it is the identity, while on proper monomorphisms and proper epimorphisms between cyclic 2-groups it is zero.  Consequently, the functor $K$ is isomorphic to the functor $H$ from Appendix \ref{pola}.}\\

In the rest of the section, we examine the cokernel functor $C$.

\begin{remark}\lb{vc}
Note that $C(\bZ/2^\ell\bZ)\simeq \bZ/2\bZ$. Indeed, a generator of $\Ext(\bZ/2^\ell\bZ,\bZ/2^\ell\bZ) \simeq \bZ/2^\ell\bZ$ has the form
\[
\bZ/2^\ell\bZ\longrightarrow\bZ/2^{2\ell}\bZ\longrightarrow\bZ/2^\ell\bZ
\]
Example \ref{lexz} shows that $\bZ/2^{2\ell}\bZ$ is not realizable as a Lagrangian extension of $\bZ/2^\ell\bZ$.  Thus the map $Lex(\bZ/2^\ell\bZ)\to\Ext(\bZ/2^\ell\bZ,\bZ/2^\ell\bZ)$ is not surjective.
Note also that any extension of the form
\begin{equation}\lb{lexz2a}
\bZ/2^\ell\bZ\longrightarrow\bZ/2^{2\ell-1}\bZ \times \bZ/2\bZ\longrightarrow\bZ/2^\ell\bZ
\end{equation}
is twice a generator of $\Ext(\bZ/2^\ell\bZ,\bZ/2^\ell\bZ) \simeq \bZ/2^\ell\bZ$.
Now, example \ref{lexz} shows that the middle term of \eqref{lexz2a} {\em is} realizable as a Lagrangian extension of $\bZ/2^\ell\bZ$, and moreover, that the image of the map $Lex(\bZ/2^\ell\bZ)\to\Ext(\bZ/2^\ell\bZ,\bZ/2^\ell\bZ)$ is $2\bZ/2^\ell\bZ$.
Thus the cokernel of the map $Lex(\bZ/2^\ell\bZ)\to\Ext(\bZ/2^\ell\bZ,\bZ/2^\ell\bZ)$ must be isomorphic to $\bZ/2\bZ$.
\end{remark}

\begin{lemma}
The functor $C$ is isomorphic to the functor $B\mapsto H^1{\left(\bZ/2\bZ,B^{\ot2}\right)}$. 
\end{lemma}
\begin{proof}
Since the image of $\phi$ coincides with $\Ext(\widehat B,B)^\tau_0$, we have a short exact sequence
\[
\xymatrix{0 \ar[r] & \Ext(\widehat B,B)^\tau_0 \ar[r]& \Ext(\widehat B,B)^\tau \ar[r] & C(B) \ar[r] & 0}
\]
Thus the homomorphism $\Ext(\widehat B,B)^\tau\to H^1(\bZ/2\bZ,B^{\ot 2})$ defined in section \ref{cge} factors through an embedding $C(B)\to H^1(\bZ/2\bZ,B^{\ot 2})$. 
By remark \ref{h1}, this is a natural transformation of additive functors in $B$. To show that it is an isomorphism, it suffices to check it for cyclic $B$. It is obvious for $B$ of odd order.
Remark \ref{h1} says that $H^1(\bZ/2\bZ,B^{\ot 2})$ is isomorphic (as an abelian group) to $\bZ/2\bZ$.
Finally, remark \ref{vc} shows that  $C(\bZ/2^\ell\bZ)$ is also isomorphic to $\bZ/2\bZ$, making the map $C(B)\to H^1(\bZ/2\bZ,B^{\ot 2})$ an isomorphism.
\end{proof}

Finally, we summarize the results of this section in the following theorem.

\begin{theorem}
For any finite abelian group $B$, the following sequence is exact:
\[
0\longrightarrow H(B)\longrightarrow Lex(B)\overset{\phi}{\longrightarrow} \Ext(\widehat{B},B)^\tau\longrightarrow H(B)\longrightarrow 0
\]
\end{theorem}

%%%%%%%%%%%%%%%%%%%%%

\appendix

\section{Lagrangian algebras in pointed categories}\lb{la}

\setcounter{subsection}{1}

Here, we describe the homomorphism $Lex(L)\to H^3(\widehat A,k^*)$ explicitly by giving a 3-cocycle $\beta\in Z^3(\widehat A,k^*)$ representing the Lagrangian extension $(A,q,\iota)$. We do it by computing the associator of the pointed category $\C(A,q)_{R}$ of modules over an etale (i.e. indecomposable separable commutative) algebra $R$.

Simple objects $I(a)$ of $\C(A,\alpha,c)$ are labeled by elements of $A$.  We fix fusion isomorphisms $\iota_{a,b}:I(a)\otimes I(b)\to I(a+b)$ for $a,b\in A$.  By going around the diagram  
\[
\xymatrix{I(a)\otimes I(b)\otimes I(c)  \ar[r]^{\iota_{a,b}\otimes1} \ar[d]_{1\otimes\iota_{b,c}} & I(a+b)\otimes I(c) \ar[d]^{\iota_{a+b,c}}\\ I(a)\otimes I(b+c) \ar[r]_{\iota_{a,b+c}} & I(a+b+c)}
\]
clockwise, we obtain an automorphism of the simple object $I\left(a+b+c\right)\in\C(A,\alpha,c)$; namely, $\alpha(a,b,c)\cdot1_{I(a+b+c)}$.  The braiding in $\C(A,\alpha,c)$ is given by \eqref{braid}.
As in section \ref{pbfc}, we write the category $\C(A,\alpha,c)$ as $\C(A,q)$, where $q\in Q(A,k^*)$ is the (unique) quadratic function corresponding to the pair $(\alpha,c)$.  

An indecomposable separable algebra $R$ in $\C(A,q)$ is supported by a subgroup $B\subset A$.
Let $R(B)\in\C(A,q)$ be the object given by 
\[
R(B)=\bigoplus_{b\in B}{I(b)}\ .
\]
The multiplication map $\mu$ on $R(B)$ has a form $\mu(b,b')=\eta(b,b')\cdot\iota_{b,b'}$ for $\eta\in C^2(B,k^*)$ such that $\d\eta=\alpha|_B$.
The coboundary condition on $\eta$ makes this multiplication associative.  
The algebra $R(B)$ is commutative if and only if $\eta(b,b')=c(b,b')\eta(b',b)$ for all $b,b'\in B$.  
In terms of the quadratic function $q$, commutativity is equivalent to isotropy of the subgroup $B\subset A$. In this case the multiplication on $R(B)$ is defined uniquely up to an isomorphism.

The free module functor gives a a tensor equivalence $\C(A,q)_{R(B)}\to\V(A/B,\beta)$ for some $\beta\in H^3(A/B,k^*)$. Here we compute the associator $\beta$ explicitly.

Let $J:\C(A,q)\to\C(A,q)_{R(B)}$ be the free module functor $J(X)=X\otimes R(B)$.  The functor $J$ is tensor with the tensor structure $J_{X,Y}:J(X)\otimes_{R(B)}J(Y)\to J{\left(X\otimes Y\right)}$ given by the composition
\[
\xymatrix{
X\otimes R(B)\otimes Y\otimes R(B)\ar[rrr]^{1_X\ot c_{R(B),Y}\ot1_{R(B)}}&\null&\null&X\ot Y\ot R(B)\ot R(B)\ar[rr]^{1_X\ot1_Y\ot\mu}&\null&X\ot Y\ot R(B)\\
}\ .
\]
We also fix an isomorphism $J(I)\simeq R(B)$ given by the right unit isomorphism in $\C(A,q)$.

All simple $R(B)$-modules are induced from simple objects of $\C(A,q)$.   For $a\in A$, denote $J{\left(I(a)\right)}$ by $J(a)$, and define a map of $R(B)$-modules $\phi_{a,b}:J(a)\otimes_{R(B)}J(b)\rightarrow J(a+b)$ as the composition:
\[
\xymatrix{
J{\left(I(a)\right)}\ot_{R(B)}J{\left(I(b)\right)}\ar[rr]^{J_{I(a),I(b)}} & & J{\left(I(a)\ot I(b)\right)}\ar[rr]^{J{\left(\iota_{a,b}\right)}} && J{\left(I(a+b)\right)}\\
}
\]
Note that the clockwise composition of the diagram 
\begin{equation}\lb{afj}%\tag{$*$}
\xymatrix{
J(a)\ot_{R(B)}J(b)\ot_{R(B)}J(c)\ar[rr]^{\phi_{a,b}\otimes1}\ar[d]_{1\ot \phi_{b,c}} && J(a+b)\ot_{R(B)}J(c)\ar[d]^{\phi_{a+b,c}}\\
J(a)\ot_{R(B)}J(b+c)\ar[rr]_{\phi_{a,b+c}} && J(a+b+c)\\
}\end{equation}
is  $\alpha(a,b,c)\cdot1_{J(a+b+c)}$.

For $\ell\in B$, define $\theta_\ell:I(\ell)\otimes R(B)\rightarrow R(B)$ by $\theta_\ell=\oplus_{\ell'\in B}{\left(\eta(\ell,\ell')\cdot\iota_{\ell,\ell'}\right)}$.  This gives an isomorphism $\theta_\ell:J(\ell)\to J(0)$ of right $R(B)$-modules.  The collection of $\theta_\ell$ allows us to define, for any $a\in A$ and any $\ell\in B$ an isomorphism of $R(B)$-modules $\vartheta_{a,\ell}:J(a+\ell)\to J(a)$ 
\[
\xymatrix{
J(a+\ell)\ar[r]^{\vartheta_{a,\ell}}&J(a)\\
J(a)\ot_{R(B)}J(\ell)\ar[u]^{\phi_{a,\ell}}\ar[r]_{1\ot\theta_{\ell}}&J(a)\ot_{R(B)}R(B)\ar@{=}[u]\\
}
\]
The collection of $\vartheta$s establishes that up to isomorphism, simple $R(B)$-modules are labeled by elements of the quotient group $A/B$.

Similarly, define $\vartheta'_{\ell,a}:J(\ell+a)\rightarrow J(a)$ by the diagram
\[
\xymatrix{
J(\ell+a)\ar[r]^{\vartheta'_{\ell,a}}&J(a)\\
J(\ell)\ot_{R(B)}J(a)\ar[u]^{\phi_{\ell,a}}\ar[r]_{\theta_{\ell}\ot1}&R(B)\ot_{R(B)}J(a)\ar@{=}[u]\\
}
\]
Commutativity of addition implies that the maps $\vartheta_{a,\ell}$ and $\vartheta'_{\ell,a}$ differ by a nonzero scalar.  More precisely, $\vartheta'_{\ell,a}=c(a,\ell)\cdot\vartheta_{a,\ell}$.

Choose a set-theoretic section $s:A/B\rightarrow A$ of the canonical projection.  Define $\gamma:A/B\times A/B\to L$ by $\gamma(x,y)=s(x)+s(y)-s(x+y)$, and define $J_{x,y}:J(s(x))\ot_{R(B)}J(s(y))\to J(s(x+y))$ as the composition 
\[
\xymatrix{J(s(x))\ot_{R(B)}J(s(y))\ar[rr]^{\phi_{s(x),s(y)}}&\null&J(s(x)+s(y))\ar[rr]^{\vartheta_{s(x+y),\gamma(x,y)}}&\null&J(s(x+y))}\ .
\]
The tensor category $\C(A,q)_{R(B)}$ of right $R(B)$-modules is equivalent to $\mathcal{V}{\left(A/B,\beta\right)}$ for some $\beta\in Z^3(A/B,k^*)$.  The clockwise composition of the arrows of the diagram
\begin{equation}\label{beta1}%\tag{\dag}
\xymatrix{
J(s(x))\ot_{R(B)}J(s(y))\ot_{R(B)}J(s(z))\ar[r]^(.55){J_{x,y}\ot1}\ar[d]_{1\ot J_{y,z}}&J(s(x)+s(y))\ot_{R(B)}J(z)\ar[d]^{J_{x+y,z}}\\
J(s(x))\ot_{R(B)}J(s(y+z))\ar[r]_{J_{x,y+z}}&J(s(x+y+z))\\
}
\end{equation}
is  $\beta(x,y,z)\cdot1_{J(s(x+y+z))}$, and so we need to compute this composition to see the value of the associator $\beta(x,y,z)$.

We now expand the diagram \eqref{beta1}.  In doing so, we will suppress the tensor product symbols.  We will also suppress labels on identity morphisms.   Moreover, we will write $R$ instead of $R(B)$ where applicable.

\begin{equation}\label{diagram}%\tag{$**$}
\xymatrix{
J(s(x))J(s(y))J(s(z))\ar[rr]^{\phi_{s(x),s(y)}1}\ar[d]_{1\phi_{s(y),s(z)}}&\null&J(s(x)+s(y))J(s(z))\ar[rr]^{\vartheta_{s(x+y),\gamma(x,y)}}\ar[d]_{\phi_{s(x)+s(y),s(z)}}&\null&J(s(x+y))J(s(z))\ar[d]^{\phi_{s(x+y),s(z)}}\\
J(s(x))J(s(y)+s(z))\ar[rr]^{\phi_{s(x),s(y)+s(z)}}\ar[d]_{1\vartheta_{s(y+z),\gamma(y,z)}}&\null&J(s(x)+s(y)+s(z))\ar[rr]^{\vartheta_{s(x+y)+s(z),\gamma(x,y)}}\ar[d]_{\vartheta_{s(x)+s(y+z),\gamma(y,z)}}&\null&J(s(x+y)+s(z))\ar[d]^{\vartheta_{s(x+y+z),\gamma(x+y,z)}}\\
J(s(x))J(s(y+z))\ar[rr]_{\phi_{s(x),s(y+z)}}&\null&J(s(x)+s(y+z))\ar[rr]_{\vartheta_{s(x+y+z),\gamma(x,y+z)}}&\null&J(s(x+y+z))
}
\end{equation}

It is straightforward to see that the clockwise reading of the upper left cell of the diagram \eqref{diagram} contributes a factor of $\alpha(s(x),s(y),s(z))$.

\begin{lemma}
The lower left cell of the diagram \eqref{diagram} contributes a factor of $\alpha(s(x),s(y+z),\gamma(y,z))^{-1}$.
\end{lemma}
\begin{proof}
Let $a = s(x), b=s(y+z) , \ell= \gamma(y,z)$.
Consider the diagram
\[
\xymatrix{
J(a)J(b+\ell)\ar[ddd]_{1\vartheta_{b,\ell}}\ar[rrr]^{\phi_{a,b+\ell}}&&&J(a+b+\ell)\ar[ddd]^{\vartheta_{a+b,\ell}}\\
&J(a)J(b)J(\ell)\ar[r]^{\phi_{a,b}1}\ar[ul]_{1\phi_{b,\ell}}\ar[d]_{11\theta_\ell}&J(a+b)J(\ell)\ar[ur]_{\phi_{a+b,\ell}}\ar[d]^{1\theta_\ell}&\\
& J(a)J(b)R\ar[r]_{\phi_{a,b}1}\ar@{=}[dl]&J(a+b)R\ar@{=}[dr]&\\
J(a)J(b)\ar[rrr]_{\phi_{a,b}}&&&J(a+b)\\
}
\]
All cells except the top one commute. By \eqref{afj}, the clockwise composition of the top cell gives $\alpha(a,b,\ell)^{-1}$.  
\end{proof}

\begin{lemma}
The upper right cell of the diagram \eqref{diagram} contributes a factor of 
\[
\frac{\alpha(\gamma(x,y),s(x+y),s(z))c(s(x+y)+s(z),\gamma(x,y))}{c(s(x+y),\gamma(x,y))}.
\]
\end{lemma}
\begin{proof}
Let  $\ell=\gamma(x,y)$, $a=s(x+y)$, and $b=s(z)$, and consider the diagram
\[
\xymatrix{
J(\ell+a)J(b)\ar[rrr]^{\vartheta'_{\ell,a}1}\ar[ddd]_{\phi_{\ell+a,b}}&&&J(a)J(b)\ar[ddd]^{\phi_{a,b}}\ar@{=}[dl]\\
\null&J(\ell)J(a)J(b)\ar[d]_{1\phi_{a,b}}\ar[r]^{\theta_\ell11}\ar[ul]^{\phi_{\ell,a}1}&RJ(a)J(b)\ar[d]^{1\phi_{a,b}}&\null\\
\null&J(\ell)J(a+b)\ar[r]_{\theta_{\ell}1}\ar[dl]_{\phi_{\ell,a+b}}&RJ(a+b)\ar@{=}[dr]&\null\\
J(\ell+a+b)\ar[rrr]_{\vartheta'_{\ell,a+b}}&&&J(a+b)\\
}
\]
The rightmost and the centre cells commute on the nose.  The leftmost cell contributes a factor of $\alpha(\ell,a,b)$.  The bottom cell is the definition of $\vartheta'_{\ell,a+b}$.  Similarly, the top cell is the definition of $\vartheta'_{\ell,a}$.  Now, in the diagram \eqref{diagram}, $\vartheta'$s do not appear.  Instead, $\vartheta$s appear.  Comparing $\vartheta'$ with $\vartheta$ gives factors of $c(a+b,\ell)$ from the bottom cell and $c(a,\ell)^{-1}$ from the top cell.  The inverse appears in the second factor due to the orientation being reversed relative to the definition of $\vartheta_{a,\ell}$.  
\end{proof}

\begin{lemma}
The lower right cell of the diagram \eqref{diagram} contributes a factor of 
\[
\frac{\eta(\gamma(x,y+z),\gamma(y,z))}{\eta(\gamma(x+y,z),\gamma(x,y))}.
\]
\end{lemma}
\begin{proof}
We split the lower right cell of the diagram \eqref{diagram} into two triangles as follows:
\begin{equation}\label{diag2}%\tag{\ddag}
\xymatrix{
J(s(x)+s(y)+s(z))\ar[rrr]^{\vartheta_{s(x+y)+s(z),\gamma(x,y)}}\ar[ddrrr]_{\vartheta_{s(x+y+z),\gamma(x,y,z)}}\ar[dd]_{\vartheta_{s(x)+s(y+z),\gamma(y,z)}}&\null&\null&J(s(x+y)+s(z))\ar[dd]^{\vartheta_{s(x+y+z),\gamma(x+y,z)}}\\
\null&\null&\null&\null\\
J(s(x)+s(y+z))\ar[rrr]_{\vartheta_{s(x+y+z),\gamma(x,y+z)}}&\null&\null&J(s(x+y+z))\\
}
\end{equation}
Let us now consider the expanded general form of such triangles:

\[
\xymatrix{
J{\left(a+\ell+\ell'\right)}\ar[rrrr]^{\vartheta_{a,\ell+\ell'}}\ar[ddddd]_{\vartheta_{a+\ell,\ell'}}&&&&J(a)\\
\null&\null&J(a)J{\left(\ell+\ell'\right)}\ar[ull]^{\phi_{a,\ell+\ell'}}\ar[r]^{1\theta_{\ell+\ell'}}&J(a)R\ar@{=}[ur]&\null\\
\null&J(a+\ell)J{\left(\ell'\right)}\ar[uul]^{\phi_{a+\ell,\ell'}}\ar[d]^{1\theta_{\ell'}}&J(a)J(\ell)J{\left(\ell'\right)}\ar[u]^{1\phi_{\ell,\ell'}}\ar[l]^{\phi_{a,\ell}1}\ar[d]^{11\theta_{\ell'}}&\null&\null\\
\null&J(a+\ell)\ar@{=}[ddl]&J(a)J(\ell)R\ar[l]^{\phi_{a,\ell}1}\ar@{=}[d]&J(a)R\ar@{=}[uuur]&\null\\
\null&\null&J(a)J(\ell)\ar[dll]^{\phi_{a,\ell}}\ar[ur]_{1\theta_\ell}&\null&\null\\
J(a+\ell)\ar[rrrr]_{1_{J(a+\ell)}}&&&&J(a+\ell)\ar[uuuuu]_{\vartheta_{a,\ell}}\\
}
\]
All of the cells in this diagram commute except for the right-centre cell, which contributes a factor of $\eta{\left(\ell,\ell'\right)}$.  Therefore, the lower triangle of \eqref{diag2} contributes a factor of $\eta(\gamma(x,y+z),\gamma(y,z))$, and the upper triangle contributes a factor of $\eta(\gamma(x+y,z),\gamma(x,y))^{-1}$.
\end{proof}

Thus we have the following:

\begin{theorem}
The category $\C(A,q)_{R(B)}$ of right modules over an etale algebra $R(B)\in\C(A,q)$ is tensor equivalent to the category $\mathcal{V}{\left(A/B,\beta\right)}$ of $A/B$-graded vector spaces, where the associator $\beta\in Z^3(A/B,k^*)$ is given by the formula
\begin{equation}\label{betadef}
\beta(x,y,z)=
\end{equation}
\[
=\frac{\alpha{\left(s(x),s(y),s(z)\right)}\alpha{\left(\gamma(x,y),s(x+y),s(z)\right)}c{\left(s(x+y)+s(z),\gamma(x,y)\right)}\eta{\left(\gamma(x,y+z),\gamma(y,z)\right)}}{\alpha{\left(s(x),s(y+z),\gamma(y,z)\right)}c{\left(s(x+y),\gamma(x,y)\right)}\eta{\left(\gamma(x+y,z),\gamma(x,y)\right)}}
\]
\end{theorem}

Recall that Lagrangian algebras in a non-degenerate braided pointed category $\C(A,q)$ correspond to Lagrangian subgroups of $A$.  For a Lagrangian $L\subset A$, the projection $\pi:A\to \widehat B$ (from section \ref{cle}) gives an isomorphism $A/B\simeq \widehat B$, which in turn gives an isomorphism $H^3(A/B,k^*)\simeq H^3(\widehat L,k^*)$. 

\begin{corollary}
The homomorphism $Lex(L)\to H^3(\widehat L,k^*)$ sends a Lagrangian extension $L\subset A$ into the class of $\beta$ given by \eqref{betadef} transported along the isomorphism $H^3(A/B,k^*)\simeq H^3(\widehat L,k^*)$. 
\end{corollary}

%%%%%%%%%%%%%%%%%%%%%

\section{Polynomial Functors}\lb{pola}

\setcounter{subsection}{1}

Let $A,B\in\mathcal{A}b$, and let $P:\mathcal{A}b\to\mathcal{A}b$ be a functor.  The diagram of canonical projections and injections
\[
\xymatrix{
A\ar@(dl,ul)[]^{1}\ar@/^15pt/[r]&A\oplus B\ar@/^15pt/[r]\ar@/^15pt/[l]& B   \ar@(dr,ur)[]_{1}  \ar@/^15pt/[l] }
\vspace{5pt}
\] 
gives rise to a (right) splitting of its image under $P$:
\[
\xymatrix{P(A\oplus B)\ar@/^15pt/[r]&P(A)\oplus P(B)\ar@/^15pt/[l]\\}
\]
Define $\delta P(A,B)$ to be the complement of this splitting, so that
\[
P(A\op B)\simeq P(A)\op P(B)\op\delta P(A,B)\ .
\]
This defines a functor $\delta P:\mathcal{A}b\times\mathcal{A}b\to\mathcal{A}b$ which we call the {\em polarisation of $P$}.\\

We call a constant functor {\em polynomial of degree zero}.  For $n>0$, a functor $P$ is said to be {\em polynomial of degree $n$} if its polarisation $\delta P$ is polynomial of degree $n-1$ in each of its arguments.

\begin{remark}
An additive functor $P:\mathcal{A}b\to\mathcal{A}b$ is polynomial of degree 1. 
Conversely, a polynomial functor $P$ of degree 1 such that $P(0)=0$ is additive.
\end{remark}

\bpr\lb{b.one}
\begin{enumerate}[1.]
\item The polarization is a functor
$$\delta(\cdot):\mathcal{F}un(\mathcal{A}b,\mathcal{A}b)\to\mathcal{F}un(\mathcal{A}b\times\mathcal{A}b,\mathcal{A}b)\ .$$
More precisely, given functors $P,Q:\mathcal{A}b\to\mathcal{A}b$ and a natural transformation $f:P\to Q$, there is a natural transformation $\delta f:\delta P\to\delta Q$.
\item If $F_1,F_2,F_3:\mathcal{A}b\to\mathcal{A}b$ are functors such that the sequence
\[
0\to F_1\to F_2\to F_3\to 0
\] is exact, then the corresponding sequence
\[
0\to \delta F_1\to \delta F_2\to \delta F_3\to 0
\]
is again exact.
\end{enumerate}
\epr
\bpf
\begin{enumerate}[1.]
\item Let $P,Q:\mathcal{A}b\to\mathcal{A}b$ be functors, and let $f:P\to Q$ be a natural transformation.
%Void ends HERE.
%%%%%%%%%%%%%%%%%%%%%%%%%%%%%%%%%%%%%%%%%%%%%%%%
Let $A,B\in\mathcal{A}b$ be given.  Then there is a morphism $f_{A\oplus B}:P(A\oplus B)\to Q(A\oplus B)$.  Note that $P(A\oplus B)\simeq P(A)\oplus P(B)\oplus \delta P(A,B)$ and $Q(A\oplus B)\simeq Q(A)\oplus Q(B)\oplus \delta Q(A,B)$.  The map $f_{A\oplus B}$ then has the form $f_{A\oplus B}=f_A\oplus f_B\oplus {(\delta f)}_{A,B}$.  Explicitly, ${(\delta f)}_{A,B}=\left.f_{A\oplus B}\right|_{\delta P(A,B)}$.  Naturality of $\delta f$ then follows from naturality of $f$.
\item\lb{b.12} Let $A,B\in\mathcal{A}b$. %Void begins HERE.
%%%%%%%%%%%%%%%%%%%%%%%%%%%%%%%%%%%%%%%%%%%%%%%%
\void{ By hypothesis, there are morphisms $\iota:F_1(A\oplus B)\to F_2(A\oplus B)$ and $\pi:F_2(A\oplus B)\to F_3(A\oplus B)$ making the sequence
\[
0\to F_1(A\oplus B)\overset{\iota}{\longrightarrow} F_2(A\oplus B)\overset{\pi}{\longrightarrow}F_3(A\oplus B)\to 0
\]
exact.  For each $i\in\{1,2,3\}$, let $f_i$ denote the natural isomorphism
$$F_i(A\oplus B)\overset{\sim}{\longrightarrow}F_i(A)\oplus F_i(B)\oplus \delta F_i(A,B)\ ,$$
let $p_i:F_i(A\oplus B)\twoheadrightarrow \delta F_i(A,B)$ be the corresponding projection map, and let $\gamma_i=\left.{f_i}^{-1}\right|_{\delta F_i(A,B)}$.  Define morphisms $\iota^\prime:\delta F_1(A,B)\to \delta F_2(A,B)$ and $\pi^\prime:\delta F_2(A,B)\to \delta F_3(A,B)$ by $\iota^\prime=p_2\circ\iota\circ\gamma_1$, $\pi^\prime=p_3\circ\pi\circ\gamma_2$.  We claim that the sequence
\[
0\to \delta F_1(A,B)\overset{\iota^\prime}{\longrightarrow} \delta F_2(A,B)\overset{\pi^\prime}{\longrightarrow} \delta F_3(A,B)\to 0
\]
is exact.  %Void ends HERE.
%%%%%%%%%%%%%%%%%%%%%%%%%%%%%%%%%%%%%%%%%%%%%%%%
}
For each $i=1,2,3$, there are splittings
\[
\xymatrix{
{F_i}(A\oplus B)\ar@/^15pt/[rr]&&{F_{i}}(A)\oplus{F_i}(B)\ar@/^15pt/[ll]\\
}\ \text{and}\ 
\xymatrix{
{F_i}(A\oplus B)\ar@/^15pt/[rr]&&{\delta F_i}(A,B)\ar@/^15pt/[ll]\\
}
\]
These splittings give rise to the diagram
\beq\lb{jkl;}
\xymatrix{
\null&0\ar[d]&0\ar[d]&0\ar[d]&\null\\
0\ar[r]&{\delta F_1}(A,B)\ar[r]\ar@/^6pt/[d]&{\delta F_2}(A,B)\ar[r]\ar@/^6pt/[d]&{\delta F_3}(A,B)\ar[r]\ar@/^6pt/[d]&0\\
0\ar[r]&{F_1}(A\oplus B)\ar[r]\ar@/^6pt/[u]\ar@/^6pt/[d]&{F_2}(A\oplus B)\ar[r]\ar@/^6pt/[u]\ar@/^6pt/[d]&{F_3}(A\oplus B)\ar[r]\ar@/^6pt/[u]\ar@/^6pt/[d]&0\\
0\ar[r]&{F_1}(A)\oplus{F_1}(B)\ar[r]\ar@/^6pt/[u]\ar[d]&{F_2}(A)\oplus{F_2}(B)\ar[r]\ar[d]\ar@/^6pt/[u]&{F_3}(A)\oplus{F_3}(B)\ar[r]\ar[d]\ar@/^6pt/[u]&0\\
\null&0&0&0&\null\\
}
\eeq
By hypothesis, the last two rows of the diagram \eqref{jkl;} are exact.  The columns of \eqref{jkl;} are split exact.  The snake lemma implies that the first rowis exact. 
\end{enumerate}
\epf

\begin{example}\lb{additive}
For an abelian group $B$, the functor $H(B)=H^1(\bZ/2\bZ,B^{\ot2})$ is additive.  Indeed, for abelian groups $B$ and $D$, the polarisation of $H$ is $\delta H(B,D) = H^1(\bZ/2\bZ,(B\ot D)\oplus(D\ot B))$.  Observe that $(B\ot D)\oplus(D\ot B)$ is the induced module ${Ind}_{0}^{\bZ/2\bZ}(B\oplus D)$ from the trivial subgroup.  Shapiro's lemma implies that 
$$H^1(\bZ/2\bZ,(B\ot D)\oplus(D\ot B))={H^1}{\left(\bZ/2\bZ,{Ind}_{0}^{\bZ/2\bZ}(B\oplus D)\right)}\simeq H^1(0,B\oplus D)=0\ ,$$ 
whence $H(B\oplus D)\simeq H(B)\oplus H(D)$, as claimed.\\
By the definition,
\[
H^1(\bZ/2\bZ,B^{\ot2})=\{\alpha\in B^{\ot2}\ |\tau(\alpha)=-\alpha\}/\{\beta-\tau(\beta)|\ \beta\in B^{\ot2}\}\ ,
\]
with $\tau$ acting as the transposition of tensor factors.  
Define a map $B_2\to H^1(\bZ/2\bZ,B^{\ot2})$ by $b\mapsto b\otimes b$ (note that the condition $2b=0$ implies that $\tau(b\ot b)=-b\ot b$).  This map is an isomorphism.  Indeed, by additivity it suffices to see that $B_2\to H^1(\bZ/2\bZ,B^{\ot2})$ is an isomorphism for a cyclic $B$.  It is obvious for $B$ of odd order, since $H^1(\bZ/2\bZ,B^{\ot 2}) = 0$ in that case.  Note that $\tau$ is identity on $B^{\ot 2}$ for a cyclic $B$. Thus for $B = \bZ/2^\ell\bZ$ we have $Z^1(\bZ/2\bZ,(\bZ/2^\ell\bZ)^{\ot2})=(\bZ/2^\ell\bZ)_2$ and $B^1(\bZ/2\bZ,(\bZ/2^\ell\bZ)^{\ot2})=0$.
Hence, we have a functorial isomorphism
$$H(B) \simeq B_2\ .$$ 

%%%%%%%%%%%%%%%%%%%%%%%%%%%%%%%%%%%%%%%%%%%%%%%%%%%%%%%%%%
%Void begins HERE.
\void{
Thus for $B$ of odd order, $H(B)=H^1(\bZ/2\bZ,B^{\ot 2}) = 0$.
For $B = \bZ/2^\ell\bZ$, the transposition $\tau$ acts trivially on $(\bZ/2^\ell\bZ)^{\ot2}$. Hence $Z^1(\bZ/2\bZ,(\bZ/2^\ell\bZ)^{\ot2})=\bZ/2\bZ$ and $B^1(\bZ/2\bZ,(\bZ/2^\ell\bZ)^{\ot2})=0$.  Thus $H(\bZ/2^\ell\bZ) = \bZ/2\bZ$.\\
Now we compute the effect of $H$ on morphisms.  By additivity, it suffices to consider the effect of $H$ on cyclic groups.  Consider a proper epimorphism $f:\bZ/2^\ell\bZ\to\bZ/2^m\bZ$.  Up to an automorphism of the target, $f$ has the form $f([n]_{2^\ell})=[n]_{2^m}$.  
Note that $f^{\ot 2}$ becomes $f$ upon the identification $C\ot C\simeq C$ for a cyclic $C$.
It follows that $H(f){\left([1]_2\right)}={\left[2^{\ell-1}\right]}_{2^m}$.  But since $\ell-1\geq m$, it follows that $2^m\mid 2^{\ell-1}$, whence $H(f){\left([1]_2\right)}=0$.  Thus $H$ takes proper epimorphisms to zero.  Now consider a proper monomorphism $g:\bZ/2^\ell\bZ\to\bZ/2^m\bZ$ for some $m>\ell$.  Up to an automorphism of the target, $g$ has the form $g{\left([1]_{2^\ell}\right)}={\left[2^{m-\ell}\right]}_{2^m}$.  It follows that $H(g){\left([1]_2\right)}={\left[2^{m-1}\right]}_{2^m}$.  Thus $H$ takes proper monomorphisms to proper monomorphisms.
}%Void ends HERE.
%%%%%%%%%%%%%%%%%%%%%%%%%%%%%%%%%%%%%%%%%%%%%%%%%%%%%%%%%%
\end{example}

\begin{example}\lb{J}
Consider the exact sequence
\begin{equation}\lb{seqB}
Hom_\bZ(B\ot B,\bQ/\bZ)\longrightarrow Q(B,\bQ/\bZ)\longrightarrow J(B)\longrightarrow 0
\end{equation}
The functor $J$ is the cokernel of the map ${Hom_{\bZ}}{\left(B^{\ot2},\bQ/\bZ\right)}\to Q(B,\bQ/\bZ)$ sending a bilinear form $\varsigma:B\times B\to\bQ/\bZ$ to the quadratic form $q(x)=\varsigma(x,x)$ for $x\in B$.\\
The functor $J$ is additive.  Indeed, part \ref{b.12} of proposition \ref{b.one} implies that the sequence
\[
{\delta Hom_{\bZ}}{\left((-)^{\ot2},\bQ/\bZ\right)}\longrightarrow {\delta Q}(-,\bQ/\bZ)\longrightarrow{\delta J}(-)\longrightarrow 0
\]
is exact. The morphism  ${\delta Hom_{\bZ}}{\left((-)^{\ot2},\bQ/\bZ\right)}\to {\delta Q}(-,\bQ/\bZ)$ between polarisations is surjective. Thus ${\delta J}=0$.\\
  Furthermore, $J(B)=0$ for $B$ of odd order.  It can be shown that $Q(\bZ/2^\ell\bZ,\bQ/\bZ)\simeq \bZ/2^{\ell+1}\bZ$ with a generator $q_\ell(x)=\frac{x^2}{2^{\ell+1}}$. Then \eqref{seqB} becomes
$$\bZ/2^\ell\bZ\hookrightarrow\bZ/2^{\ell+1}\bZ\twoheadrightarrow J(\bZ/2^\ell\bZ)\ ,$$
whence $J(\bZ/2^\ell\bZ)\simeq\bZ/2\bZ$ as an abelian group.\\
Now we compute the effect of $J$ on morphisms. Obviously $J$ takes isomorphisms between cyclic $2$-groups to the identity.  
Consider a proper epimorphism of cyclic groups $f:\bZ/2^\ell\bZ\to\bZ/2^m\bZ$.  Up to an automorphism of the target, $f$ has the form $f{\left([1]_\ell\right)}=[1]_m$.  
The induced map $f^*:Q(\bZ/2^m\bZ,\bQ/\bZ)\to Q(\bZ/2^\ell\bZ,\bQ/\bZ)$ is given by ${f^*}{\left(q_m\right)}([1]_\ell)=q_m(f([1]_\ell))=\frac{1}{2^{m+1}}=\frac{2^{\ell-m}}{2^{\ell+1}}$, whence ${f^*}{\left(q_m\right)}=2^{\ell-m}q_\ell$.  It follows that the induced map between groups of quadratic forms is zero, so $J$ sends proper epimorphisms to zero.\\
Up to an automorphism of the target, a proper monomorphism $g:\bZ/2^\ell\bZ\to\bZ/2^n\bZ$ is $g{\left([1]_\ell\right)}=\left[2^{n-\ell}\right]_n$.  The induced map $g^*:Q(\bZ/2^\ell\bZ,\bQ/\bZ)\to Q(\bZ/2^n\bZ,\bQ/\bZ)$ has the form ${g^*}{\left(q_n\right)}([1]_\ell)=q_n(g([1]_\ell)={q_n}{\left({\left[2^{n-\ell}\right]}_{n}\right)}=\frac{2^{2n-2\ell}}{2^{n+1}}=\frac{2^{n-\ell}}{2^{\ell+1}}$, so that $g^*(q_n)=2^{n-\ell}q_\ell$.  The induced map ${g^*}$ is zero, so $J$ maps proper monomorphisms to zero.
\end{example}

%%%%%%%%%%%%%%%%%%%%%

%%%%%%%%%%%%


\begin{thebibliography}{11}

\bibitem{bre}
L. Breen.  ``On the functorial homology of abelian groups''. {\em Journal of Pure and Applied Algebra} \textbf{142} (1999), 199--237.
%
\bibitem{da}
A. Davydov.  ``Bogomolov multiplier, double class-preserving automorphisms and modular invariants for orbifolds'. \emph{Journal of Mathematical Physics} \textbf{55} (2014), 9,  arXiv:1312.7466v3.
%
\bibitem{dmno} 
A. Davydov, M. M\"uger, D. Nikshych, V. Ostrik. ``Witt group of non-degenerate braided fusion categories''.  \emph{Journal f\"ur die Reine und Angewandte Mathematik} \textbf{677} (2013), 135-177. arXiv:1009.2117. 
%
\bibitem{das1}
A. Davydov, D. Simmons. ``On Lagrangian algebras in group-theoretical braided fusion categories''.  \emph{Journal of Algebra} \textbf{471} (2017), 149--175.  arXiv:1603.04650.
%
\bibitem{ds}
M. Bischoff, A. Davydov, D. Simmons.  ``Automorphisms of Lagrangian algebras''.  In preparation.
%
\bibitem{em}
S. Eilenberg, S. Mac Lane. ``On the groups $H(\Pi,n)$ I, II''.  \emph{Ann. of Math.} \textbf{58} (1953), 55--106; \textbf{70} (1954), 49--137.
%
\bibitem{egno}
P. Etingof, S. Gelaki, D. Nikshych, V. Ostrik. \emph{Tensor Categories}, American Mathematical Society, vol. 205 (2015). 344 pp.
%
\bibitem{gjs}
P. Grossman, D. Jordan, N. Snyder. ``Cyclic extensions of fusion categories via the Brauer-Picard groupoid''. \emph{Quantum Topology}, \textbf{Vol. 6} (2015), no. 2, pp. 313-331.
%
\bibitem{js}
A. Joyal, R. Street. ``Braided tensor categories''. \emph{Adv. Math.} \textbf{102} (1993), no. 1, 20--78.
%
\bibitem{ki}
A. Kirillov, Jr. ``Modular categories and orbifold models I, II''.  \emph{Commun. Math. Phys.} \textbf{229}  (2002), 309-335.
%
\bibitem{lkw}
T. Lan, L. Kong, X.-G. Wen.  ``Modular extensions of unitary braided fusion categories and $2+1$D topological/SPT orders with symmetries.''  arXiv:  1602.05936v1.
%
\bibitem{nr}
D. Nikshych, B. Riepel.  ``Categorical Lagrangian Grassmannians and Brauer-Picard groups of pointed fusion categories''. \emph{Journal of Algebra}, \textbf{Vol. 411, 1} (2014),  191--214, arXiv:1309.5026.
%
\bibitem{pa}
B. Pareigis.  ``On braiding and dyslexia''.  \emph{Journal of Algebra} \textbf{171} (1995), no. 2, 413--425.
%
\end{thebibliography}
\end{document}